\documentclass{amsart}

\usepackage{enumitem}
\usepackage{amssymb}
\usepackage{tikz-cd} 
\usepackage{hyperref}
\hypersetup{colorlinks=true}
\usepackage{dsfont}
\usepackage{mdframed}
\usepackage{caption}
\usepackage[toc,page]{appendix}
\usetikzlibrary{decorations.pathmorphing}

\usetikzlibrary{decorations.markings}
\tikzset{negated/.style={
        decoration={markings,
            mark= at position 0.5 with {
                \node[transform shape] (tempnode) {$\backslash$};
            }
        },
        postaction={decorate}
    }
}

\theoremstyle{plain}
\newtheorem{theorem}{Theorem}[section]

\newtheorem{proposition}[theorem]{Proposition}
\newtheorem{lemma}[theorem]{Lemma}
\newtheorem{corollary}[theorem]{Corollary}

\newtheorem*{theorem*}{Theorem}

\theoremstyle{definition}
\newtheorem{definition}[theorem]{Definition}
\newtheorem*{definition*}{Definition}

\theoremstyle{remark}
\newtheorem{remark}{Remark}[section]

\newtheorem{question}{Question}

\DeclareMathOperator{\Aut}{Aut}

\DeclareMathOperator{\GL}{\mathbf{GL}}

\DeclareMathOperator{\Comm}{Comm}

\title{The definitions of approximate lattices}
\author{Simon Machado\\
ETHZ}
\email{smachado@ethz.ch}

\begin{document}

\begin{abstract}
Approximate lattices of locally compact groups were first studied in a seminal monograph of Yves Meyer and were subsequently used in the theory of aperiodic order to model objects such as Pisot numbers, quasi-cristals or aperiodic tilings. Meyer studied approximate lattices of Euclidean spaces - now dubbed Meyer sets. A fascinating feature of this theory is the wealth of natural and equivalent definitions of Meyer sets that are available. 

The study of approximate lattices in non-commutative groups has recently attracted more attention. They are meant as generalisation of lattices of locally compact groups (i.e. discrete subgroup of finite co-volume) and, as such,  are roughly defined as approximate subgroups of locally compact groups that are both discrete and have finite co-volume. In this generalized framework, the correct notion of `finite co-volume' is however up for interpretation. A number of definitions recently arose following work of Bj\"{o}rklund and Hartnick, Hrushovski, and the author. In stark contrast with the Euclidean framework,  little is known of the relations between these notions.

Our main objective in this paper is to address this issue. We establish clear relations between these definitions and show that, in fact, these definitions are not all equivalent. We prove nevertheless that all the notions mentioned are related to one another and simply represent varying degrees of generality. Among the original results we prove here, we extend a cornerstone theorem of Meyer and show that all strong approximate lattices are laminar; we build measures with weak invariance properties on the invariant hull of approximate lattices; we use quasi-models to show that any BH-approximate lattice is contained in an approximate lattice; we prove that even laminar approximate lattices can be arbitrarily far from strong approximate lattices and model sets. 
\end{abstract}

\maketitle

\section{Introduction}
An \emph{approximate subgroup} is defined as a subset $\Lambda$ of a group $G$ that is symmetric ($\Lambda = \Lambda^{-1}$), contains the identity and satisfies $\Lambda^2 \subset F\Lambda$ for some $F \subset G$ finite. Meyer was the first to study approximate lattices in his seminal monograph \cite{meyer1972algebraic}. There - and in the large body of work that followed \cite{moody1997meyer} - the focus was  more specifically put on \emph{uniform approximate lattices}. They are defined as those discrete approximate subgroups $\Lambda$ of a locally compact group $G$ that are \emph{relatively dense}: every point of $G$ is within a uniformly bounded distance of a point of $\Lambda$ or, equivalently, $\Lambda K = G$ for some compact subset $K$ of $G$.  Meyer related uniform approximate lattices of locally compact abelian groups to the so-called \emph{model sets} which are often understood as the `ideal' or `regular' approximate lattices. Model sets are subsets of $G$ built from a lattice $\Gamma$ in a product of locally compact groups $G \times H$. One first cuts a strip of $\Gamma$ along $G$ and then projects it to $G$. Precisely, fix a relatively compact symmetric neighbourhood of the identity $W_0$ of $H$ - the \emph{window} - and define the model set
$$ M(G,H,\Gamma,W_0):=p_G\left(\Gamma \cap (G \times W_0)\right)$$ 
where $p_G: G \times H \rightarrow G$ denotes the natural projection. 

 Understanding the structure of uniform approximate lattices already offers a formidable challenge, but Meyer's model sets beg for a definition of \emph{non-uniform} approximate lattices. Perhaps the simplest attempt at a definition is due to Hrushovski \cite[Def. A.1]{hrushovski2020beyond}: $\Lambda$ is an \emph{approximate lattice} if it is a uniformly discrete approximate subgroup of a locally compact group $G$ admitting a subset $\mathcal{F} \subset G$ with finite Haar-measure such that $\Lambda \mathcal{F} = G$. This slick definition is satisfying in many respect but does not capture \emph{a priori} the dynamical aspects a generalisation of lattices should. Alternative definitions of approximate lattices - which pre-date Hrushovski's - with an ergodic-theoretic flavour were introduced in \cite{bjorklund2016approximate, machado2020apphigherrank}. They are the \emph{strong}, $\star$- and \emph{BH-approximate lattices}. All three definitions involve relevant probability measures on the \emph{invariant hull} of $\Lambda$ which is defined as the closure of the orbit $G \cdot \Lambda$ of $\Lambda$ in the Chabauty space of closed subsets of $G$ endowed with the Chabauty--Fell topology. The invariant hull is a dynamical space generalising the quotient space by a closed subgroup. When $\Lambda$ is a subgroup all these definitions reduce to $\Lambda$ being a lattice, see \S \ref{Section: The definitions of approximate lattices} for details.
 
 In \cite{bjorklund2016approximate}, \cite{hrushovski2020beyond} and \cite{machado2020apphigherrank} the respective authors discussed the advantages of each notion. It was furthermore asked in Caprace and Monod's problem list \cite{zbMATH06949681} which should be the correct definition of an approximate lattice in a (non-commutative) locally compact group. Our main result addresses this problem by establishing all relations between the six definitions introduced above. Together with previously known results we can draw the following schema:

\begin{figure}[h]
\begin{tikzcd}[column sep=tiny] \text{Model sets} \ar[d, leftrightsquigarrow] \ar[drr,leftrightsquigarrow]\\
\text{Strong a.l.} \arrow[rr, Rightarrow] &&\ \star- \text{a.l.} \    \arrow[drr, "/" marking, leftsquigarrow]\arrow[rr, Rightarrow] && \text{Approximate lattices} \arrow[d,Leftarrow]\arrow[rr, Rightarrow]\arrow[rr, leftarrow, dashed, bend left=40]  && \text{BH-a.l.} \\
&&&& \text{Uniform a.l.} 
\end{tikzcd}
\begin{center}
$\protect\Rightarrow$: implication, $\protect\rightsquigarrow$: commensurability and $\protect\dashrightarrow$: inclusion.
\end{center}
\caption{Hierarchy of approximate lattices.} \label{Figure: Hierarchy of approximate lattices}
\end{figure}
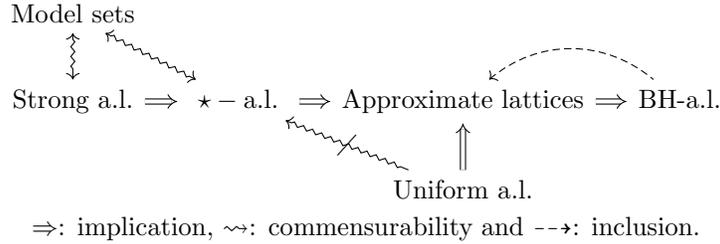

More precisely, the original results we obtain here are: 

\begin{theorem}\label{Proposition: Detail resultat principal}
Let $G$ be a second countable locally compact group.
\begin{enumerate}
\item Let $\Lambda$ be a $\star$-approximate lattice in $G$. Then $\Lambda^8$ contains a model set.
\item There is an approximate lattice that is not commensurable with a $\star$-approximate lattice.
\item An approximate lattice in $G$ is a BH-approximate lattice.
\item Let $\Lambda$ be a BH-approximate lattice in $G$. Then $\Lambda$ is contained in (but not necessarily commensurable with) an approximate lattice $\Lambda^*$ in $\overline{\Comm_{G}(\Lambda)}$. 

\end{enumerate}
\end{theorem}

The proof of Theorem \ref{Proposition: Detail resultat principal} intertwines in an essential way the combinatorial tools arising from combinatorial and model-theoretic ideas developed for instance in \cite{hrushovski2020beyond, hrushovski2019amenability,  machado2019goodmodels, MR3345797} and the ergodic-theoretic tools from \cite{bjorklund2016approximate, machado2020apphigherrank}. This paper is the first one - to the knowledge of the author - where these arguments are so tightly knit together. As such, it offers a panorama of the methods encountered while studying approximate lattices and we hope it will help create a bridge between the two languages.

Let us briefly comment Theorem \ref{Proposition: Detail resultat principal}.  Part (1) is a generalisation of Meyer's theorem \cite{meyer1972algebraic} valid for $\star$-approximate lattices in all locally compact second countable groups. This complements recent papers working towards generalisations of Meyer's theorem \cite{bjorklund2016approximate,hrushovski2020beyond,mac2023structure,machado2020apphigherrank} and achieves a complete generalisation for $\star$-approximate lattices. The proof relies on the combination of two tool-sets.  It first uses transverse measure on cross-sections to establish that $\Lambda$ satisfies an amenability-like condition.  It is then combined with an extension of the Massicot--Wagner argument for actions in the spirit of \cite{hrushovski2019amenability}.  

Part (2) shows that approximate lattices span more commensurability classes than $\star$-approximate lattices.  Part (2) is then a simple consequence of the existence of non-laminar approximate lattices \cite{hrushovski2020beyond} and Part (1). From a slightly tweaked construction,  examples with finer properties can be exhibited, we refer the interested reader to Proposition \ref{Proposition: laminar app lattice not close to a strong app lattice} below. 

Part (3) creates a bridge between approximate lattices in the sense of Hrushovski and the measure theory of the invariant hull, showing that approximate lattices are in particular BH-approximate lattices.  Notice that approximate lattices are not defined via a condition on the invariant hull and it is \emph{a priori} not clear how to relate the two notions. We use \emph{periodization maps} - a tool coming from the Siegel transform \cite{10.2307/1969027} first used in \cite{bjorklund2016approximate} that counts points of a random translate of $\Lambda$ in a given Borel subset -, the existence of a `fundamental domain' and a Hahn--Banach argument to build measures on the hull of an approximate lattice.  

Part (4) provides a partial converse to part (3).  The proof relies on the notion of quasi-models and the proof of the extension of Meyer's theorem for semi-simple Lie groups due to Hrushovski \cite[Thm. 7.4]{hrushovski2020beyond} and an argument about joinings coming from \cite{bjorklund2019borel}.

Concerning the structure of approximate lattices up to commensurability, we believe that Figure \ref{Figure: Hierarchy of approximate lattices} indicates that Hrushovski's definition of approximate lattices is the most suitable for investigations.  Nevertheless, even when investigating the structure of approximate lattices up to commensurability, strong approximate lattices show up naturally, see \cite{mac2023structure}.

\subsection{Laminar approximate lattices far from model sets}
Even for approximate subgroups commensurable to actual lattices, the various notions represent various levels of generality. 

\begin{proposition}\label{Proposition: laminar app lattice not close to a strong app lattice}
Let $\Gamma$ be a lattice in a rank one simple Lie group and $m \geq 0$ be an integer. There is an approximate lattice $\Lambda \subset \Gamma$ such that $\Lambda^m$ is not a $\star$-approximate lattice. In particular, $\Lambda^k$ is not a strong approximate lattice for all $1 \leq k \leq m$. 
\end{proposition}

Proposition \ref{Proposition: laminar app lattice not close to a strong app lattice} is particularly striking in the light of the following heuristic from additive combinatorics: given an approximate subgroup $A$, each power $A^{n+1}$ has more structure than $A^n$ - see Lemma \ref{Lemma: Intersection of approximate subgroups} for an elementary instance of that phenomenon. This heuristic transcribes particularly well for approximate lattices in \emph{amenable groups}: if $\Lambda$ is an approximate lattice and $G$ is amenable, then $\Lambda^4$ already contains a model set \cite{machado2019goodmodels}. Proposition \ref{Proposition: laminar app lattice not close to a strong app lattice} shows that this fails dramatically in non-amenable ambient groups and can be loosely rephrased by stating that there are laminar approximate lattices as far as one wants from a model set. Our proof however crucially uses the rank one assumption. We therefore raise the following question: 

\begin{question}
Let $G$ be a simple Lie group of rank at least $2$. Is there $m > 0$, such that for every approximate lattice $\Lambda \subset G$, $\Lambda^m$ contains a model set?
\end{question}

\subsection{Outline of the paper}

Our paper starts with recording a number of preliminary notions and results. In \S \ref{Section: The definitions of approximate lattices} we discuss elementary properties of approximate subgroups, good models and quasi-models, we introduce the notion of uniform discreteness and define the invariant hull and properties pertaining to it. In \S \ref{Section: The definitions of approximate lattices} we go on and provide in detail all the definitions of approximate lattices along with a survey of elementary results useful later on. We prove Theorem \ref{Proposition: Detail resultat principal} (1) in \S \ref{Subsection: Laminarity of strong approximate lattices and transverse measures}.  We then show Theorem \ref{Proposition: Detail resultat principal} (3) in \ref{Subsection: Invariant hull, strong approximate lattices and star-approximate lattices}. In \S \ref{Section: Quasi-models of uniformly discrete approximate subgroups},  we prove Theorem \ref{Proposition: Detail resultat principal} (4).  Finally, in \S \ref{Subsection: Consequences of the existence of non-laminar approximate lattices} we establish both Theorem \ref{Proposition: Detail resultat principal} (2) and Proposition \ref{Proposition: laminar app lattice not close to a strong app lattice}. 

\subsection{Acknowledgements}
I am grateful to Tobias Hartnick and Michael Bj\"{o}rklund for interesting discussions about various parts of this work.  This material is based upon work supported by the National Science Foundation under Grant No. DMS-1926686.

\section{Preliminaries}\label{Section: The definitions of approximate lattices}

\subsection{Approximate subgroups and commensurability}
A notion closely related to that of approximate subgroups is commensurability. 

\begin{definition}\label{Definition: Commensurability}
Let $X,Y \subset G$ be two subsets of a group. We say that: 
\begin{enumerate}
\item $X$ is \emph{covered by finitely many left (resp. right) translates of} $Y$ if there is $F \subset G$ finite such that $X \subset FY$ (resp. $X \subset YF$);
\item $X$ is ($K$-)commensurable with $Y$ if there is $F \subset G$ finite ($|F| \leq K$) such that $X \subset FY \cap YF$ and $Y  \subset FX \cap XF$;
\item $g \in G$ \emph{commensurates} $X$ if $gXg^{-1}$ and $X$ are commensurable. Then $\Comm_G(X) \subset G$ denotes the subgroup of elements commensurating $X$;
\item $\alpha \in \Aut(G)$ \emph{commensurates} $X$ if $\alpha(X)$ and $X$ are commensurable.
\end{enumerate}
\end{definition}

We collect now two useful results concerning approximate subgroups and commensurability. 

\begin{lemma}[e.g. \cite{machado2019goodmodels}]\label{Lemma: Intersection of commensurable sets}
Take $X,Y_1,\ldots Y_n$ subsets of a group $G$. Assume that there exist $F_1, \ldots, F_n \subset G$ finite such that $X \subset F_iY_i$ for all $i \in \{1,\dots, n\}$. Then there is $F' \subset X$ with $|F'| \leq |F_1|\cdots |F_n|$ such that $$X \subset F'\left( Y_1^{-1}Y_1 \cap \cdots \cap Y_n^{-1}Y_n\right).$$ 
\end{lemma}

\begin{lemma}[e.g. \cite{machado2019goodmodels}]\label{Lemma: Intersection of approximate subgroups}
Let $\Lambda_1, \ldots, \Lambda_n$ be $K_1,\ldots,K_n$-approximate subgroups of some group. We have: 
\begin{enumerate}
\item if $k_1, \ldots, k_n \geq 2$, then there is $F$ with $|F|\leq K_1^{k_1-1} \cdots K_n^{k_n-1}$ such that $$\Lambda_1^{k_1} \cap \cdots \cap \Lambda_n^{k_n} \subset F\left(\Lambda_1^2 \cap \cdots \cap \Lambda_n^2\right);$$ 
\item if $k_1, \ldots, k_n \geq 2$, then $\Lambda_1^{k_1} \cap \cdots \cap \Lambda_n^{k_n}$ is a $K_1^{2k_1-1}\cdots K_n^{2k_n-1}$-approximate subgroup;
\item if $\Lambda_1', \ldots, \Lambda_n'$ is a family of approximate subgroups such that $\Lambda_i'$ is commensurable with $\Lambda_i$ for all $1 \leq i \leq n$, then $\Lambda_1'^2 \cap \cdots \cap \Lambda_n'^2$ is commensurable with $\Lambda_1^2\cap \cdots \cap \Lambda_n^2$.
\end{enumerate}
\end{lemma}

\subsection{Models and quasi-models}
We introduce two key notions of the theory of approximate subgroups that will play a central role in what follows. 
\subsubsection{Good models}
Good models are related to a notion of regularity for approximate subgroups. 

\begin{definition}
An approximate subgroup $\Lambda$ of a group $G$ has a \emph{good model} if there is a group homomorphism $f: \langle \Lambda \rangle \rightarrow H$ with target a locally compact group such that $f(\Lambda)$ is relatively compact and $f^{-1}(U) \subset \Lambda$ for some neighbourhood $U$ of $e \in H$. We say that $\Lambda$ is \emph{laminar} if it is commensurable with an approximate subgroup that has a good model. 
\end{definition}

We refer the interested reader to \cite{hrushovski2019amenability}, \cite{{MR3345797}} and \cite{machado2019goodmodels} for accounts and surveys regarding good models.  The following criterion for the existence of good models is folklore, see e.g. \cite[Prop. 1.2, Lem. 3.3]{machado2019goodmodels}.

\begin{lemma}\label{Lemma: Charac. good models}
Let $\Lambda$ be an approximate subgroup of some group. Then $\Lambda$ has a good model if and only if there are approximate subgroups $\Lambda_n$ for all $n \geq 0$ such that $\Lambda_n$ is commensurable with $\Lambda$, $\Lambda_{n+1}^2 \subset \Lambda_n$ and $\Lambda_0=\Lambda$.
\end{lemma}

The next result relates good models and the model sets defined in the introduction:

\begin{lemma}[\cite{machado2019goodmodels}]\label{Lemma: good models and model sets}
Let $\Lambda$ be an approximate lattice in a locally compact group $G$. Then $\Lambda$ has a good model if and only if it contains a model set. 
\end{lemma}

 \subsubsection{Quasi-homomorphisms as models}\label{Subsubsection: Quasi-homomorphisms as models}
 Unfortunately, not all approximate subgroups have good models. Counter-examples were found in \cite{hrushovski2019amenability}, \cite{machado2019goodmodels} and \cite{hrushovski2020beyond} and all involved the notion of \emph{quasi-morphism} in an essential way. A deep result of Hrushovski asserts that this is not a coincidence.
 
 \begin{theorem}[Quasi-model theorem, Thm 4.1 \cite{hrushovski2020beyond}]\label{Theorem: Hrushovski's quasi-models}
 Let $\Lambda$ be an approximate subgroup of some group. There is a quasi-homomorphism $f: \langle \Lambda \rangle \rightarrow H$ and a compact normal subset $K \subset H$ such that: 
 \begin{enumerate}
 \item $\{f(\gamma_1\gamma_2)f(\gamma_2)^{-1}f(\gamma_1)^{-1} : g,h \in \gamma\} \subset K$;
 \item for all relatively compact neighbourhoods of the identity $W \subset H$, $\Lambda$ is commensurable with $f^{-1}(WK)$.

 \end{enumerate}
 Moreover, by restricting $f$ to a finite index subgroup $\Gamma_0 \subset \Gamma$ we may assume that $A:=\overline{\langle K \rangle}$ is isomorphic to a Euclidean space on which $G$ acts by isometries and $f(\Gamma_0)A$ is dense in $H$. We will call such a map a \emph{quasi-model (of $(\Lambda,\Gamma)$)}.
 \end{theorem}
 
 We note here that Hrushovski obtains the additional structure on $H$ by studying the subgroups $A$ topologically generated by a normal compact subset. He calls these subgroups \emph{rigid} and we will encounter them in this work as well. 
 
Since the defect $K$ of $f$ is contained in $A$, the subset $f(\Gamma_0)A$ is a subgroup. This can be used to derive a number of properties of $f$. However, considering the subset $f(\Gamma_0)A$ goes against the advantage of considering quasi-models as it erases completely the defect of $f$. Hence, we adopt another viewpoint. We suppose that $K \subset A$ is a Euclidean ball - which can always be arranged by augmenting the size of $K$ harmlessly - and consider $f(\Gamma_0)K$. Since $f$ is a quasi-homomorphism and $K$ is normal $(f(\Gamma_0)K)^2 \subset f(\Gamma_0)K^3$, which is covered by finitely many right translates of $f(\Gamma_0)K$. So we observe that $f(\Gamma_0)K$ is an approximate subgroup of $H$. This enables us to prove: 

\begin{lemma}\label{Lemma: Preparing quasi-models}
Let $\Lambda$ be an approximate subgroup of some group $\Gamma$ that commensurates it. There is a finite index subgroup $\Gamma_0 \subset \Gamma$ and a quasi-model $f: \Gamma_0 \rightarrow H$ of $\Lambda^2 \cap \Gamma_0$ with defect contained in a compact normal subset $K$ such that $A:=\overline{\langle K \rangle}$ is isomorphic to a Euclidean space on which $G$ acts by isometries and $K \subset A$ is a Euclidean ball about the identity, $f(\Gamma_0)$ is relatively dense in $H$ and $\overline{f(\Gamma_0)K}$ projects surjectively to $H/A$.
\end{lemma}

\begin{proof}
First of all, note that the map obtained by the composition $\bar{f}$ of $f$ and the natural projection $H \rightarrow H/A$ is a group homomorphism. By \cite[Thm 5.16]{hrushovski2020beyond} we can assume that $\overline{\bar{f}(\Lambda)}$ has non-empty interior and $\bar{f}(\Gamma_0)$ is dense. Write $\Xi$ the approximate subgroup $f(\Gamma_0)K$. We already know that $\Xi A$ is dense in $H$. The projection of $\overline{f(\Lambda)K}$ to $H/A$ contains $\overline{\bar{f}(\Lambda)}$ and, hence, has non-empty interior. Therefore, the projection of $\overline{\Xi}$ is a dense subgroup with non-empty interior i.e. is equal to $H/A$. Now, take $W \subset H$ a symmetric relatively compact neighbourhood of the identity and write $\Xi_W:= \Xi^2 \cap W^2A$. Then $\Xi_W$ is an approximate subgroup (Lemma \ref{Lemma: Intersection of approximate subgroups}) commensurated by $\Xi$ and contained in a compact neighbourhood of $A$. According to Lemma \ref{Lemma: Schreiber around subspace} below, there are a vector subspace $A_1 \subset A$ normalised by $\Xi$ and a compact subset $C_1 \subset H$ such that $\Xi_W \subset C_1A_1$ and $A_1 \subset C_1\Xi_W$. Write $A_2$ the space orthogonal to $A_1$ and notice that both $A_1$ and $A_2$ are normalised by $\Xi$, hence by $H$ since $\overline{\Xi A}$ is equal to $H$. Now, 
$$W^2A_2 \cap \Xi^2 \subset \Xi_W \subset C_1A_1.$$
So $W^2A_2 \cap \Xi^2$ is relatively compact. Since for every relatively compact subset $W' \subset H$, $W'A$ is covered by finitely many translates of $WA$, we have by Lemma \ref{Lemma: Intersection of commensurable sets} that $W'^2A_2 \cap \Xi^2$ is relatively compact as well. Let $f':\Gamma_0 \rightarrow H/A_2$ denote the composition of $f$ with the natural projection $H \rightarrow H/A_2$. Then $f'$ is a quasi-morphism with defect contained in the projection of $K$ to $H/A_2$. Take any neighbourhood of the identity $W'' \subset H/A_2$, then our discussion implies that $f'^{-1}(W'')$ is contained in $f^{-1}(C_2)$ for some relatively compact subset $C_2$ of $H$. Hence, $f'^{-1}(W'')$ is covered by finitely many translates of $\Lambda$.  Furthermore,  the projection of $\Xi \cdot \Xi_W$ is relatively dense in $H/A_2$.  Since $\Xi \cdot \Xi_W$  is covered by finitely many translates of $\Xi$, the projection of $f(\Gamma_0)$ (i.e. $f'(\Gamma_0)$) is relatively dense in $H/A_2$. So $f'$ is the quasi-model we are looking for. 
\end{proof}

The following technical lemma used in the proof above will be established in the appendix. 
 
  \begin{lemma}\label{Lemma: Schreiber around subspace}
 Let $\Lambda$ be an approximate subgroup of a second countable locally compact group $G$. Let $ A\subset G$ be a closed normal subgroup isomorphic to $\mathbb{R}^n$ and suppose that there is a compact subset $K$ such that $\Lambda$ is contained in $KA$. Then there are a vector subspace $V \subset A$ and a compact subset $K'$ in the normaliser of $V$ such that $\Lambda \subset VK'$ and $V \subset \Lambda K'$. Moreover, $\Comm_{G}(\Lambda)$ normalises $V$. 
 \end{lemma}

\subsection{Uniform discreteness}
 Our next task is to introduce useful notions of discreteness.

\begin{definition}\label{Definition: Notions of discreteness}
A subset $X$ of locally compact group $G$ is \emph{left-uniformly discrete} (resp. \emph{right-uniformly discrete}) if there is an open subset $U \subset G$ such that $X^{-1}X \cap U=\{e\}$ (resp. $XX^{-1} \cap U=\{e\}$). We say that $X$ is \emph{locally finite} if for all compact subsets $K \subset G$, $|X \cap K| < \infty$. 
\end{definition}

When \emph{left-} or \emph{right-} is not specified we will always mean the left- version of uniform discreteness. When $X$ is a subset of $(G,d)$ a metric group together with a left-invariant distance, then $$\inf_{x_1\neq x_2 \in X} d(x_1,x_2)\geq r > 0$$
is equivalent to $X^{-1}X \cap B_d(e,r) =\{e\}$ where $B_d(e,r)$ denotes the open ball of radius $r$ centred at $e$. In this setup, uniform discreteness is then equivalent to the fact that any two distinct elements of $X$ are separated by at least some uniform distance. We mention now another characterisation of uniform discreteness. 

\begin{lemma}\label{Lemma: Characterisation uniform discreteness}
Let $X$ be a subset of a locally compact group. For $W \subset G$ neighbourhood of the identity, the multiplication map $X \times W \rightarrow G$ is injective if and only if $X^{-1}X \cap WW^{-1} =\{e\}$.
\end{lemma}

\begin{proof}
Suppose first that $X^{-1}X \cap WW^{-1} =\{e\}$. If $x_1,x_2 \in X$ and $ w_1, w_2 \in W$ are such that $x_1w_1=x_2w_2$, then $x_2^{-1}x_1=w_2w_1^{-1}$. So $x_1=x_2$ and $w_1=w_2$. 

Conversely, suppose that the multiplication map $X \times W \rightarrow G$ is injective for some open neighbourhood of the identity $W$. Then for all $x_1,x_2 \in X$ and for all  $w_1, w_2 \in W$ with $x_1\neq x_2$ we have $x_1w_1\neq x_2w_2$ i.e. $x_2^{-1}x_1\neq w_2w_1^{-1}$. So $x_2^{-1}x_1 \notin WW^{-1}$ and $X^{-1}X \cap WW^{-1}=\{e\}$. 
\end{proof}

Finally, we prove that for approximate subgroups the notion of uniform discreteness persists even after taking powers. 

\begin{lemma}\label{Lemma: powers of uniformly discrete approximate subgroups are uniformly discrete}
Let $\Lambda$ be an approximate subgroup of a locally compact group $G$. If $\Lambda$ is uniformly discrete, then $\Lambda^n$ is uniformly discrete for all integers $n \geq 0$. 
\end{lemma}

We leave the proof of this observation to the interested reader.

\subsection{Chabauty space and dynamical hulls of discrete subsets} Let $G$ be a locally compact second countable group and let $\mathcal{C}(G)$ be the set of closed subsets of $G$. The \emph{Chabauty-Fell} topology on $\mathcal{C}(G)$ is defined by the subbase of open subsets:
$$ U^V = \{ F \in \mathcal{C}(G) : F \cap V \neq \emptyset \}$$
 and,
$$ U_K = \{ F \in \mathcal{C}(G) : F \cap K = \emptyset \} $$
for all $V \subset G$ open and $ K \subset G$ compact. The \emph{Chabauty space} of $G$ is defined as $\mathcal{C}(G)$ endowed with the Chabauty--Fell topology. One can check that the map 
\begin{align*}
 G \times \mathcal{C}(G) & \rightarrow \mathcal{C}(G) \\
 (g, F) & \mapsto gF
\end{align*}
defines a continuous action of the group $G$ on $\mathcal{C}(G)$ and that $\mathcal{C}(G)$ is a compact metrizable set, see \cite{MR139135}. Convergence in the Chabauty space can also be characterised as follows (e.g. \cite[Section 2.2]{bjorklund2019borel}). Namely, a sequence $(F_i)_{i\geq 0}$ converges to $F \in \mathcal{C}(G)$ if and only if:
\begin{enumerate}
\item for every $x \in F$ there are $x_i \in F_i$ for all $i \in \mathbb{N}$ such that $x_i \rightarrow x$ as $i \rightarrow \infty$; 
\item If $x_i \in F_i$ for all $i \in \mathbb{N}$ then every accumulation point of $(x_i)_{i \geq 0}$ lies in $F$.
\end{enumerate}    

\begin{definition}\label{Definition: Invariant hull}
Let $ F \in \mathcal{C}(G)$.  The \emph{invariant hull} is 
$$\Omega_F:= \overline{\{gF : g \in G\}}$$
the closure of the $G$-orbit of $F$. The \emph{extended invariant hull} is 
$$\Omega_F^{ext}:=  \{ X \in \mathcal{C}(G): X^{-1}X \subset \overline{F^{-1}F} \}.$$
\end{definition}

Both spaces are related in the following way. 

\begin{lemma}[\cite{machado2020apphigherrank}]\label{Lemma: Extended invariant hull contains invariant hull}
  Let $X_0$ be a closed subset of a locally compact second countable group $G$. Then $\Omega_{X_0,G}^{ext}$ is a closed subset of $\mathcal{C}(G)$ stable under the $G$-action. Thus, the set $\Omega_{X_0,G}^{ext}$ is a metrizable compact continuous $G$-space. Moreover, the invariant hull $\Omega_{X_0}$ is contained in $\Omega_{X_0,G}^{ext}$. 
 \end{lemma}
 
The structure of the space $\Omega_F$ is closely linked to the structure of $F$ as we will see multiple times later. For now, we offer the following examples:
\begin{enumerate}
\item $ \emptyset \notin \Omega_F$ if and only if $F$ is relatively dense \cite[Prop. 4.4]{bjorklund2016approximate};
\item if $H$ is a closed subgroup, then $\Omega_{H}$ is isomorphic as a compact $G$-space to the one-point compactification of $G/H$ \cite[Lem. 2.3]{bjorklund2019borel};
\item if $F' \subset \Omega_F$, then $F'^{-1}F' \subset \overline{F^{-1}F}$ \cite[Lem 4.6]{bjorklund2016approximate} (see also the proof of Lemma \ref{Lemma: Extended invariant hull contains invariant hull}).
 \end{enumerate}
 
Approximate lattices are thought of as generalisations of lattices, and it can sometimes be enlightening to look at what some of the objects we study become in the case of lattices. Recall that a lattice $\Gamma$ of a locally compact group $G$ is a discrete subgroup such that $G/\Gamma$ admits a $G$-invariant Borel probability measure. In other words, $\Gamma$ is a lattice if and only if $\Omega_{\Gamma}$ admits a $G$-invariant (for the left-action) Borel probability measure $\nu$ such that $\nu(\emptyset)=0$. This last condition will appear again naturally and was first considered in this context in \cite{bjorklund2016approximate}.  

 \begin{definition}\label{Definition: Non-trivial measures}
  We say that a Borel measure $\nu$ on $\mathcal{C}(G)$ is \emph{proper} if $\nu(\{\emptyset\})=0$.
 \end{definition}
 
Restricting our attention to proper measures enables us to think only of measures that are meaningful regarding the structure of the set we consider. Indeed, for any $X \in \mathcal{C}(G)$, the Dirac mass $\delta_{\emptyset}$ at $\{\emptyset\}$ is a non-proper $G$-invariant Borel probability measure on $\Omega_X^{ext}$ and there is no hope to draw much information on $X$ from it.

\section{The definitions of approximate lattices}

As must already be apparent from the introduction, an approximate lattice is generally speaking an approximate subgroup that is both discrete and has finite co-volume. We have already defined and investigated in the previous sections the notions of approximate subgroup and uniform discreteness; it remains to define an appropriate notion of ``finite co-volume". This is precisely at that point that the definitions of approximate lattices that we will introduce differ: some approaches are set-theoretic, others dynamical and the last ones aim to relate our object to a lattice. The goal of this section is to introduce all the aforementioned notions of approximate lattices and survey some elementary results. 

Before we move on, we note that all definitions of approximate lattices defined below coincide when the approximate subgroup considered is in fact a subgroup. In other words, all definitions of approximate lattices generalise the \emph{actual} lattices, albeit with varying degrees of generality.
\subsubsection{Approximate lattices and uniform approximate lattices} 

We have already introduced approximate lattices and uniform approximate lattices earlier in the introduction. Nevertheless, we provide definitions again.  

\begin{definition}[Meyer, \cite{meyer1972algebraic}; Bj\"{o}rklund--Hartnick, \cite{bjorklund2016approximate}]
Let $G$ be a locally compact group. We say that $\Lambda \subset G$ is a \emph{uniform approximate lattice} if $\Lambda$ is a uniformly discrete approximate subgroup of $G$ and there is a compact subset $K \subset G$ such that $\Lambda K = G$. 
\end{definition}

By replacing te compactness condition on $K$ with a weaker finite volume assumption we have:

\begin{definition}[Hrushovski, A.1, \cite{hrushovski2020beyond}]\label{Definition: Approximate lattice}
Let $G$ be a locally compact group. We say that $\Lambda \subset G$ is an \emph{approximate lattice} if $\Lambda$ is a uniformly discrete approximate subgroup of $G$ and there is $ \mathcal{F} \subset G$ with finite left-Haar measure such that $\Lambda \mathcal{F} = G$. 
\end{definition}

\begin{remark}\label{Remark: Lack of symmetry}
Let us briefly mention a few remarks about (lack of) symmetry in the definitions of uniform approximate lattices and approximate lattices. Since $\Lambda$ is symmetric, $\Lambda K = G$ implies $K^{-1} \Lambda = G$ and $\Lambda$ is thus both left and right relatively dense. Moving $\mathcal{F}$ around in Definition \ref{Definition: Approximate lattice} is not so harmless. Using the symmetry of $\Lambda$ we only prove that there is $\mathcal{F} \subset G$ with $\mathcal{F}\Lambda = G$ for some $\mathcal{F}$ with finite \emph{right}-Haar measure. One can therefore wonder: what of those uniformly discrete approximate subgroups that admit a subset $\mathcal{F}$ with finite left-Haar measure such that $\mathcal{F}\Lambda =G$? Since $G$ is not assumed unimodular these two classes of discrete approximate subgroups could be distinct. In addition, the methods developed in \cite[Appendix A]{hrushovski2020beyond} fail for this new notion. Fortunately, the results below implicitly show that both notions coincide (e.g. Proposition \ref{Proposition: Reformulation approximate lattices}). We do not treat this case explicitly to not burden this work with one more definition.  
\end{remark}

\subsubsection{Strong, $\star$- and BH-approximate lattices}\label{Subsection: Invariant hull, strong approximate lattices and star-approximate lattices}
We turn now to the kinds of approximate lattices with a dynamical flavour. We start with strong approximate lattices. 
  
\begin{definition}[Bj\"{o}rklund--Hartnick, Definition 4.9, \cite{bjorklund2016approximate}]\label{Definition: Strong approximate lattice}
 A uniformly discrete approximate subgroup $\Lambda$ of a second countable locally compact group $G$ is a \emph{strong approximate lattice} if there is a proper $G$-invariant Borel probability measure $\nu$ on $\Omega_{\Lambda}$. 
\end{definition}

 We can relax slightly the definition of a strong approximate lattice in order to define the more malleable $\star$-approximate lattices (see \cite[ \S 2.1.4]{machado2020apphigherrank} for a discussion about the difference).

 \begin{definition}\label{Definition: star-approximate lattice}
 A uniformly discrete approximate subgroup $\Lambda$ of a second countable locally compact group $G$ is a $\star$-\emph{approximate lattice} if there is a proper $G$-invariant Borel probability measure $\nu$ on $\Omega_{\Lambda}^{ext}$. 
 \end{definition}

 By considering the support of an ergodic measure on $\Omega_{\Lambda}^{ext}$ one can see that Definition \ref{Definition: star-approximate lattice} admits the following equivalent form: 
 
 \begin{proposition}[Prop. 1, \cite{machado2020apphigherrank}]\label{Proposition: Equivalent definitions of star-approximate lattices}
  Let $\Lambda$ be a uniformly discrete approximate subgroup of a locally compact second countable group $G$. The following are equivalent: 
  \begin{enumerate}
   \item the approximate subgroup $\Lambda$ is a $\star$-approximate lattice;
   \item there is $X_0 \subset \Lambda^2$ with $X_0^{-1}X_0 \subset \Lambda^2$ such that there is a proper $G$-invariant ergodic Borel probability measure $\nu$ on $\Omega_{X_0}$ the invariant hull of $X_0$ in $G$.
  \end{enumerate}
 \end{proposition}

Finally, we can define the more general of the three `dynamically minded' definitions of approximate lattices. An \emph{admissible} measure of a locally compact group $G$ is a Borel probability measure $\mu$ on $G$ that is absolutely continuous with respect to the Haar measure and whose support generates $G$ as a semi-group. A Borel probability measure $\nu$ on $\Omega_{\Lambda}$ is called $\mu$-stationary if $\mu *  \nu = \nu$ (e.g.\cite[\S 4.2]{bjorklund2016approximate}).

 \begin{definition}[Bj\"{o}rklund--Hartnick, Definition 4.12, \cite{bjorklund2016approximate}]\label{Definition: BH-approximate lattice}
 A uniformly discrete approximate subgroup $\Lambda$ of a second countable locally compact group $G$ is a \emph{BH-approximate lattice} if for all admissible measures $\mu$ on $G$ there is a proper $\mu$-stationary measure $\nu$ on $\Omega_{\Lambda}$. 
 \end{definition}

 \subsubsection{Cut-and-project schemes and model sets}\label{Subsubsection: Cut-and-project schemes and model sets}

Finally, we (re-)introduce the model sets. They represent what an ideal approximate lattice should look like. They originate from lattices in products of groups, which are now known to be extremely regular. We start with the definitions of cut-and-project schemes.

\begin{definition}[Meyer, \cite{meyer1972algebraic}; Bj\"{o}rklund--Hartnick, \cite{bjorklund2016approximate}]\label{Definition: Cut-and-project scheme}
A triple $(G,H,\Gamma)$ is a \emph{cut-and-project scheme} if: 
\begin{enumerate}
\item $G$ and $H$ are locally compact groups;
\item $\Gamma \subset G \times H$ is a lattice; 
\item $\Gamma$ projects injectively to $G$ and densely to $H$.
\end{enumerate}
\end{definition}

Cut-and-project schemes enable us to efficiently build interesting families of approximate lattices: 

\begin{definition}[Meyer, \cite{meyer1972algebraic}; Bj\"{o}rklund--Hartnick, \cite{bjorklund2016approximate}]\label{Definition: Model sets}
Let $(G,H, \Gamma)$ be a cut-and-project scheme. Take $W_0 \subset H$ a symmetric relatively compact neighbourhood of the identity and let $p_G:G \times H \rightarrow G$ denote the natural projection. We call the subset 
$$ M:= p_G\left(\left(G \times W_0\right) \cap \Gamma\right)$$ 
a \emph{model set (associated to $(G,H,\Gamma)$)}.
\end{definition}

As we will mention later, model sets are approximate lattices in all meanings defined above, and this justifies in large parts our interest in this construction, see \S \ref{Subsection: Wrapping up the proof}.

\subsection{Elementary results}
We survey now a number of relatively elementary properties of approximate lattices. Since these results are now well-known and the proofs are elementary, we provide only the proofs that we deem valuable for an illustrative purpose. 

\subsubsection{Stability with subsets and powers}

We first make an easy observation about the powers of approximate lattices. 

\begin{lemma}\label{Lemma: powers of approximate lattices}
Let $\Lambda$ be an approximate subgroup of a locally compact group $G$. If $\Lambda$ is an approximate lattice (resp. a uniform approximate lattice, resp. $\star$-approximate lattice, resp. BH-approximate lattice), then $\Lambda^n$ is an approximate lattice (resp. a uniform approximate lattice, resp. $\star$-approximate lattice, resp. BH-approximate lattice) for every integer $n \geq 1$. Conversely, if there is an integer $n \geq 1$ such that $\Lambda^n$ is an approximate lattice (resp. a uniform approximate lattice, resp. BH-approximate lattice), then $\Lambda$ is an approximate lattice (resp. a uniform approximate lattice, resp. BH-approximate lattice).
\end{lemma}

We emphasize that both directions of Lemma \ref{Lemma: powers of approximate lattices} fail for strong approximate lattices and the converse statement fails for $\star$-approximate lattices. We point to the discussions regarding this issue in \cite[\S 2.1.4]{machado2020apphigherrank} for instance.

\subsubsection{Fundamental domains}
Hrushovski proved a precision on the type of subsets $\mathcal{F}$ one can hope to find in the definition of approximate lattices (Definition \ref{Definition: Approximate lattice}). When $\Lambda$ is a lattice (i.e. $\Lambda$ is an approximate lattice that is closed under multiplication) it is a standard fact that one may find $\mathcal{F}$ as in Definition \ref{Definition: Approximate lattice} such that the family of sets $(\lambda \mathcal{F})_{\lambda \in \Lambda}$ forms a partition of the ambient group $G$. Such a set is often called a \emph{fundamental domain}, see \cite{raghunathan1972discrete} and references therein. While this is not possible when $\Lambda$ is a general approximate lattice, one can still choose $\mathcal{F}$ satisfying weaker disjointness properties: 

\begin{lemma}[Hrushovski, Prop.  A.2, \cite{hrushovski2020beyond}]\label{Lemma: ersatz fundamental domain}
Let $\Lambda$ be an approximate lattice in a second countable locally compact group $G$. Let $\mu_G$ denote a (left-)Haar measure on $G$. There is $\mathcal{F} \subset G$ Borel such that:
\begin{enumerate}
\item $\mu_G(\mathcal{F}) < \infty$;
\item $\Lambda^2 \mathcal{F} =G$;
\item the multiplication map $\Lambda \times \mathcal{F} \rightarrow G$ is one-to-one.
\end{enumerate}
\end{lemma}

The proof is based on the following observation: 

\begin{lemma}[Proof of Prop.  A.2, \cite{hrushovski2020beyond}]\label{Lemma: Bound fundamental domain}
Let $X$ be a countable subset of a locally compact group $G$. Let $\mu_G$ be a left-Haar measure. Let $B_1,B_2 \subset G$ be two Borel subsets such that $XB_1=G$ and the multiplication map $X^{-1}\times B_2 \rightarrow G$ is one-to-one. Then $\mu_G(B_2) \leq \mu_G(B_1)$. 
\end{lemma}

\begin{proof}
Since $XB_1=G$ we have $$\mu_G(B_2) \leq \sum_{x \in X}\mu_G(B_2 \cap xB_1) = \sum_{x \in X}\mu_G(x^{-1}B_2 \cap B_1).$$ But the subsets $x^{-1}B_2$ for $x \in X$ are pairwise disjoint. So we get $$\mu_G(B_2) \leq\sum_{x \in X}\mu_G(x^{-1}B_2 \cap B_1)= \mu_G(B_1 \cap \bigcup_{x \in X}x^{-1}B_2) \leq \mu_G(B_1).$$ 
\end{proof}

\begin{proof}[Proof of Lemma \ref{Lemma: ersatz fundamental domain}.]
 Let $V$ be a compact neighbourhood of the identity such that $VV^{-1} \cap \Lambda^2 =\{e\}$ and let $(g_n)_{n \geq 0}$ be a sequence of elements of $G$ such that $G=\bigcup_{n \geq 0} Vg_n$. Define inductively $B_n:=Vg_n \setminus \bigcup_{m<n} \Lambda^2B_m$ and $B:=\bigcup_{n\geq 0}B_n$. Since $\Lambda B_i\cap \Lambda B_j = \emptyset$ for $i \neq j$  and the multiplication map $\Lambda\times V \rightarrow G$ is one-to-one, we have that the multiplication map $\Lambda \times B \rightarrow G$ is one-to-one. Moreover, we have that $Vg_n \subset B_n \cup \bigcup_{m<n} \Lambda^2B_m$ for all $n \geq 0$. So $\Lambda^2B = G$. Since $\Lambda$ is an approximate lattice and according to Lemma \ref{Lemma: Bound fundamental domain}, we have that $\mu_G(B)$ is finite.
\end{proof}

 Lemma \ref{Lemma: Bound fundamental domain} also implies that two such choices of $\mathcal{F}$ must have comparable volume. Thus providing a weak notion of co-volume. 
 
 \subsubsection{Unimodularity}
One of the first results on strong approximate lattices and uniform approximate lattices is the unimodularity of their envelope, see \cite[Theorem 1.11]{bjorklund2016approximate}. Hrushovski also proved a unimodularity theorem for approximate lattices which can be easily deduced from Lemma \ref{Lemma: Bound fundamental domain}.
  
  \begin{proposition}[Hrushovski, Appendix A, \cite{hrushovski2020beyond}] \label{Proposition: Envelope of approximate lattice is unimodular}
 Let $G$ be a locally compact second countable group and suppose that $G$ contains an approximate lattice $\Lambda$. Then $G$ is unimodular.  
 \end{proposition}

  Under the conditions of Proposition \ref{Proposition: Reformulation approximate lattices} below (for instance when $\Lambda$ is a $\star$-approximate lattice) we can also prove a unimodularity statement. 
  
  \begin{proposition}\label{Proposition: Envelope of a star-approximate lattice is unimodular}
  Let $G$ be a locally compact second countable group. Suppose that $G$ contains a uniformly discrete approximate subgroup such that $\Omega_\Lambda^{ext}$ admits a measure satisfying the conclusions of Proposition \ref{Proposition: Reformulation approximate lattices}. Then $G$ is unimodular.
  \end{proposition}
  
  We delay its proof to \S \ref{Subsection: The Periodization Map}; while in spirit very similar to the proof of Proposition \ref{Proposition: Envelope of approximate lattice is unimodular}, we need to introduce more tools. Both unimodularity theorems will be absolutely necessary below to establish Proposition \ref{Proposition: Detail resultat principal}. This is related to the lack of symmetry alluded to in Remark \ref{Remark: Lack of symmetry}.

\subsubsection{Finite co-volume and commensurability}

We express now a simple and useful commensurability criterion in terms of uniform discreteness and finite co-volume.

 \begin{lemma}[ Lemma A.4, \cite{hrushovski2020beyond}]\label{Corollary: Towers of star-approximate subsets are commensurable}
  If $\Lambda_1 \subset \Lambda_2$ are two approximate lattices in a locally compact second countable group $G$, then $\Lambda_1$ and $\Lambda_2$ are commensurable.
 \end{lemma}

The proof relies on a simple counting argument involving Lemma \ref{Lemma: Bound fundamental domain} and we refer the interested reader to \cite[Appendix A]{hrushovski2020beyond}.

\section{Laminarity of $\star$-approximate lattices and transverse measures}\label{Subsection: Laminarity of strong approximate lattices and transverse measures}

In this subsection, we study the link between strong approximate lattices and model sets, eventually proving Theorem \ref{Proposition: Detail resultat principal}. Our approach is based on the combination of a dynamical viewpoint developed in \cite{BjorklundHartnickKarasik} and a combinatorial argument due to Massicot--Wagner \cite{MR3345797} inspired by results of Sanders \cite{MR2911137} and Croot--Sisask \cite{MR2738997} and recently extended in \cite{hrushovski2019amenability}. 

Let $G$ be a locally compact second countable group $G$ and consider a strong approximate lattice $\Lambda \subset G$. Fix a proper $G$-invariant Borel probability measure $\nu$ on $\Omega_{\Lambda}$. Given any element $g \in G$ define the \emph{extended canonical transverse} as 
$$\tilde{\mathcal{T}}_g:=\{X \in \Omega_{\Lambda} :  X\cap g \Lambda \neq \emptyset\}.$$ 
It satisfies two useful properties: 
 \begin{enumerate}
 \item the subset $G \cdot \mathcal{T}_g$ is Borel and has full measure in $\Omega_{\Lambda}^{ext}$ - it is in fact equal to $\Omega_{\Lambda}^{ext} \setminus \{\emptyset\}$; 
 \item there is a neighbourhood of the identity $W \subset G$ such that the obvious map $W \times \mathcal{T}_g \rightarrow \Omega_{\Lambda}^{ext}$ is one-to-one.
 \end{enumerate} 
 
  More generally, any subset of $\Omega_{\Lambda}$ satisfying (1) and (2) is a \emph{cross-section}, see (e.g. \cite{BjorklundHartnickKarasik}). The canonical transverse captures a lot of the dynamics on the space $\Omega_{\Lambda}$. This is reflected in the remarkable fact that when $\Omega_\Lambda$ is equipped with a $G$-invariant measure we are able to restrict this measure to the canonical transverse in a meaningful way. 

\begin{lemma}[Existence of transverse measures, \S 4, \cite{BjorklundHartnickKarasik}]\label{Lemma: Existence of transverse measures, strong approximate lattices}
Let $\Lambda$ be a  $\star$-approximate lattice in a second countable locally compact group $G$. Fix $\nu$ a proper $G$-invariant Borel probability measure on $\Omega_{\Lambda}$. For every $g  \subset G$ there is a finite Borel measure $\eta_g$ defined on $\tilde{\mathcal{T}}_g$ such that for all $B \subset \tilde{\mathcal{T}}_g$ and $h,g' \in G$ such that $hB \subset \tilde{\mathcal{T}}_{g'}$ we have $\eta_{g'}(hB)=\eta_{g}(B)$. 
\end{lemma}

Lemma \ref{Lemma: Existence of transverse measures, strong approximate lattices} asserts that $\nu$ can indeed be restricted to the canonical transverse associated with $g \in G$. It also explains how the invariance of $\nu$ impacts the family of measures $(\eta_g)_{g\in G}$. 

Let us now move on to the combinatorial part of this subsection. 

\begin{lemma}[Massicot--Wagner-type argument, consequence of Prop. 2.11, \cite{hrushovski2019amenability}]\label{Lemma: Massicot--Wagner for actions}
Let $\Lambda$ be a $\star$-approximate lattice in a locally compact second countable group $G$. Suppose that there are a measured set $(X,m)$ with $m$ finitely additive and equipped with a measure preserving action of $\langle \Lambda \rangle$, and a Borel subset $B \subset X$ such that $m(B) > 0 $, $\Lambda \cdot B$ is also measurable, $m(\Lambda\cdot B) < \infty$ and the \emph{stabiliser} $S:=\{\gamma \in \langle \Lambda \rangle : \gamma(\Lambda\cdot B) \cap (\Lambda\cdot B) \neq \emptyset\}$ is covered by finitely many translates of $\Lambda$. Then $S$ has a good model. 
\end{lemma}

In the following lemma, we show that a certain notion of \emph{stabilizer} tailored to our situation is generic (see \cite[Prop. 2.11]{hrushovski2019amenability} and \cite[\S 1 and proof of Thm. 12]{MR3345797} for analogous discussions in their respective frameworks). 
 \begin{lemma}\label{Lemma: measurably repetitive sets have large stabilizers}
Let $\Lambda$ be a $\star$-approximate lattice in a locally compact second countable group $G$. Let $X\subset \Lambda$ be such that $\Omega_X$ admits a proper $G$-invariant Borel probability measure. Suppose that $\Lambda^2 \times W \rightarrow G$ is one-to-one where $W$ denotes a neighbourhood of the identity.  Then $\Lambda$ is covered by finitely many left-translates of the subset $$
 S(X):=\{g \in G: \exists \nu \text{proper $G$-inv. Borel prob. measure on} \Omega_{Xg \cap X} \}.$$
  Moreover, $S(X)$ is an approximate lattice. 
 \end{lemma}
 
 \begin{proof}
 Fix a proper $G$-invariant Borel probability measure $\nu$ on $\Omega_X$. Let $m \in \mathbb{N}$ and $g_1, \ldots, g_m \in \Lambda$ be such that $\Omega_{Xg_i \cap Xg_j}$ admits no proper $G$-invariant Borel probability measure for all $i\neq j$. The natural map 
 \begin{align*}
 \phi_{ij}:\Omega_X &\longrightarrow \Omega_{Xg_i \cap Xg_j} \\
 Y &\longmapsto Yg_i \cap Yg_j
 \end{align*}
 is $G$-invariant and Borel. So $\phi_{ij}^*\nu$ is a $G$-invariant Borel probability measure on $\Omega_{Xg_i \cap Xg_j}$. Hence, $\phi_{ij}^*\nu=\delta_{\emptyset}$. 
 Now, for every $U \subset G$ open we have 
 \begin{align*}
  \mu_G(UW)/\mu_G(W) & \geq \sum_{y \in X\Lambda}\mathbf{1}_U(y) \\ 
  &\geq  \sum_{y \in \bigcup_i Xg_i^{-1}}\mathbf{1}_U(y) \\
  &\geq \sum_i \sum_{y \in  Xg_i^{-1}}\mathbf{1}_U(y) - \sum_{i<j} \sum_{y \in  Xg_i^{-1} \cap Xg_j^{-1}}\mathbf{1}_U(y).
  \end{align*}
  Integrating against $\nu$ and using $\phi_{ij}^*\nu=\delta_{\emptyset}$ we find 
  $$\mu_G(UW)/\mu_G(W) \geq \sum_{i=1}^m \mu_G(Ug_i) = m \mu_G(U)$$
 where the equality is a consequence of the unimodularity of $G$ (Proposition \ref{Proposition: Envelope of a star-approximate lattice is unimodular}).  So $m \leq \mu_G(W)^{-1}\mu_G(UW)/\mu_G(U) < \infty$. 
  
  Hence, if $m$ is chosen maximal to begin with, then for any $\lambda  \in \Lambda$ there is $i \leq m$ such that $\Omega_{X\lambda^{-1} \cap Xg_i}$ admits a proper $G$-invariant Borel probability measure. Then $g_i\lambda \in S(X)$. Since $\lambda$ was chosen arbitrary, we have $\Lambda \subset \bigcup_i g_i^{-1}S(X)$.
  
  Now $\Lambda$ is an approximate lattice \cite[Prop. 3]{machado2020apphigherrank} (see also the converse part of Proposition \ref{Proposition: Reformulation approximate lattices}).  We have that $\Lambda$ is covered by finitely many translates of $S(X)$ according to the first part of the proof. But $S(X) \subset X^{-1}X \subset \Lambda^2$. So $S(X)$ is commensurable with $\Lambda$ and contained in $\Lambda^2$. It is therefore an approximate subgroup commensurable with $\Lambda$ and, thus, an approximate lattice as well. 
 \end{proof}
 
 \begin{remark}
 A similar result may be established for \emph{repetitive} sets. Namely a set  $X \subset G$ is repetitive if $\Omega_X$ is a minimal $G$-space. Then it is well known that $\tilde{S}(X):=\{g \in G : Xg \cap X \neq \emptyset\}$ is relatively dense and $Xg \cap X$ is either empty or repetitive as well. In a sense, Lemma \ref{Lemma: measurably repetitive sets have large stabilizers} indicates that the existence of a proper $G$-invariant Borel probability measure with full support on $\Omega_X$ is a measure-theoretic alternative to the notion of repetitivity. We keep this discussion short for the sake of brevity.
 \end{remark}
 
\begin{proof}[Proof of Lemma \ref{Lemma: Massicot--Wagner for actions}.]
Let $\mathcal{E}$ be the set of all subsets $X$ of $\Lambda$ such that $\Omega_X$ admits a proper $G$-invariant Borel probability measure. We know that if $X \in \mathcal{E}$ and $g \in S(X)$ then $X \cap Xg \in \mathcal{E}$. Moreover, $\mathcal{E}$ is not empty as some element of $\Omega_\Lambda^{ext}$ belongs to  $\mathcal{E}$ (Proposition \ref{Proposition: Equivalent definitions of star-approximate lattices}). Choose $n \geq 0$ and take $X_n \in \mathcal{E}$ such that $$m(X_n^{-1}\cdot B) \leq (1 + 1/n)\inf_{X \in \mathcal{E}} m(X^{-1} \cdot B).$$ 
If $g \in S(X_m)$ we have 
\begin{align*}
m(g^{-1}X_n^{-1}\cdot B \cap X_n^{-1}\cdot B) &\geq m((g^{-1}X_n^{-1} \cap X_n^{-1})\cdot B) \\
         								 &\geq (1+1/n)^{-1}m(X_n^{-1}\cdot B)
\end{align*}
where we have used $X_ng \cap X_n \in \mathcal{E}$ to go from the first line to the second line. Thus, for all $g_1, \ldots ,g_{n-1} \in S(X_n)$ we have 
$$m((g_1\cdots g_{n-1})^{-1}X_n^{-1}\cdot B \cap X_n^{-1}\cdot B) \geq \frac{1}{n}m(X_n^{-1}\cdot B) >0.$$
In particular, $(g_1\cdots g_n)^{-1}X_n^{-1}\cdot B \cap X_n^{-1}\cdot B \neq \emptyset$ which in turn implies $(g_1\cdots g_n)^{-1} \in S$. Since $S(X_n)$ is  an approximate lattice by Lemma \ref{Lemma: measurably repetitive sets have large stabilizers}, we have that $S$ has a good model as a consequence of Lemma \ref{Lemma: Charac. good models} and Lemma \ref{Corollary: Towers of star-approximate subsets are commensurable}. 
\end{proof}

Lemma \ref{Lemma: Massicot--Wagner for actions} generalises an argument of Massicot and Wagner to the situation where the group and the space on which it acts are distinct. Note also that the scope of \cite[Prop. 2.11]{hrushovski2019amenability} is slightly different than that of Lemma \ref{Lemma: Massicot--Wagner for actions} but the proof strategies are the same. 

\begin{proof}[Proof of Theorem \ref{Proposition: Detail resultat principal}.]
Let $(\eta_g)_{g \in G}$ be the family of measures provided by Lemma \ref{Lemma: Existence of transverse measures, strong approximate lattices}. Define $\tilde{\mathcal{T}}_{\langle \Lambda \rangle}:= \bigcup_{\gamma \in \langle \Lambda \rangle} \tilde{\mathcal{T}}_\gamma$. The subset $\tilde{\mathcal{T}}_{\langle \Lambda \rangle} \subset \Omega_{\Lambda}$ is Borel and invariant under the action of $\langle \Lambda \rangle$. For all Borel subsets $B \subset \tilde{\mathcal{T}}_{\langle \Lambda \rangle}$, define a finitely additive measure $m$ by $$m(B):= \sum_{\gamma \in \langle \Lambda \rangle} \eta_{\gamma} (B_{\gamma})$$ where $(B_{\gamma})_{\gamma \in \langle \Lambda \rangle}$ is a partition of $B$ into Borel subsets with $B_{\gamma} \subset\tilde{\mathcal{T}}_{\gamma}$ for all $\gamma \in \langle \Lambda \rangle$. We claim that $m(B)$ is independent of the choice of $(B_{\gamma})_{\gamma \in \langle \Lambda \rangle}$. Indeed, if $(B_{\gamma}')_{\gamma \in \langle \Lambda \rangle}$ is another partition of $B$ with $B_{\gamma}' \subset \tilde{\mathcal{T}}_{\gamma}$ for all $\gamma \in \langle \Lambda \rangle$, then 
\begin{align*}
\sum_{\gamma \in \langle \Lambda \rangle} \eta_{\gamma}(B_{\gamma}) & = \sum_{\gamma_1,\gamma_2 \in \langle \Lambda \rangle} \eta_{\gamma_1}(B_{\gamma_1} \cap B_{\gamma_2}') \\
																  & = \sum_{\gamma_1,\gamma_2 \in \langle \Lambda \rangle} \eta_{\gamma_2}(B_{\gamma_1} \cap B_{\gamma_2}') \\
																  & = \sum_{\gamma \in \langle \Lambda \rangle} \eta_{\gamma}(B_{\gamma}') 
\end{align*}
where we have used the invariance properties of Lemma \ref{Lemma: Existence of transverse measures, strong approximate lattices} with $g=\gamma_1, g'=\gamma_2$ and $h=e$ to go from the first to the second line. Similarly, one obtains using the invariance properties of $(\eta_g)_{g \in G}$ that $m$ is left-invariant. We also have $m(\tilde{\mathcal{T}}_\gamma)= \eta_\gamma(\tilde{\mathcal{T}}_\gamma) > 0$ for all $\gamma \in \langle \Lambda \rangle$.

 Now, if $F$ is any finite subset such that $\Lambda^2 \subset F\Lambda$, $$\Lambda \cdot \tilde{\mathcal{T}}_e \subset \{ X \in \Omega_\Lambda^{ext}: X \cap \Lambda^2 \neq \emptyset\} \subset F \cdot \tilde{\mathcal{T}}_e.$$  So $$m(\Lambda \cdot \tilde{\mathcal{T}}_e) \leq m(F \cdot \tilde{\mathcal{T}}_e) \leq |F|m(\tilde{\mathcal{T}}_e) $$ by $\langle \Lambda \rangle$-invariance. Since $m(\tilde{\mathcal{T}}_e)=\eta_e(\tilde{\mathcal{T}}_e)$ is neither $0$ nor $\infty$, it remains only to prove that the stabiliser of $\Lambda \cdot \tilde{\mathcal{T}}_e$ is covered by finitely many translates of $\Lambda$ to be able to invoke Lemma \ref{Lemma: Massicot--Wagner for actions} and conclude. But $\Lambda \cdot \tilde{\mathcal{T}}_e \subset \{ X \in \Omega_{\Lambda}^{ext} : \Lambda^2 \cap X \neq \emptyset \}$. So for all $X \in \Lambda \cdot \mathcal{T}_e$ there is $\lambda \in \Lambda^2 \cap X$ which yields $\lambda^{-1}X \subset X^{-1}X \subset \Lambda^2$ (\cite[Lemma 4.6]{bjorklund2016approximate}). So $X \subset \Lambda^4$. Thus, $\gamma (\Lambda \cdot \mathcal{T}_e) \cap (\Lambda \cdot \mathcal{T}_e) \neq \emptyset$ implies $\gamma \in \Lambda^8$. Hence, $\Lambda^8$ has a good model by Lemma \ref{Lemma: Massicot--Wagner for actions}.
\end{proof}

\section{The periodization map and counting on fundamental domains}\label{Subsection: The Periodization Map}
 Our goal now is to reformulate the definition of approximate lattices (Definition \ref{Definition: Approximate lattice}) in terms of a relevant measure on the hull. This is the key to being able to compare approximate lattices and strong, $\star$- and BH-approximate lattices.

 Given a locally compact group $G$ and a closed subgroup $H$ one may use the quotient map $I$ that sends a continuous function $f: G \rightarrow \mathbb{R}$ with compact support to the function $I(f): G/H \rightarrow \mathbb{R}$ defined by $I(f)(gH) := \int_H f(gh)d\mu_H(h)$ to relate the Haar measure on $G$ and the Haar measure on $G/H$ (see e.g. \cite[Section I]{raghunathan1972discrete}). In our situation as well, we can define a map called the \emph{periodization} map. 
 
\begin{definition}\label{Definition: Periodization map}
  Let $X_0$ be a uniformly discrete subset of a locally compact second countable group $G$. Then the \emph{periodization} map defined by 
  \begin{align*}
   \mathcal{P}_{X_0}: \   C^0_c(G) & \longrightarrow C^0(\Omega_{X_0}^{ext}) \\
                 f       &\longmapsto \left( X \mapsto \sum_{x \in X} f(x) \right)
  \end{align*}
 is well-defined, continuous and left-equivariant. Here, $C^0_c(G)$ denotes the set of continuous functions on $G$ with compact support and $C^0(\Omega_{X_0}^{ext})$ is the set of continuous functions on $\Omega_{X_0}^{ext}$. Moreover, any function in the image of $\mathcal{P}_{X_0}$ has support contained in $\Omega_{X_0}^{ext}\setminus \{\emptyset\}$.
 \end{definition}

 The periodization map was first used to study approximate lattices by Bj\"{o}rklund and Hartnick in order to prove that envelopes of approximate lattices are unimodular (see \cite[\S 5]{bjorklund2016approximate}). It originates in the celebrated work of Siegel on the so-called \emph{Siegel transform} \cite{10.2307/1969027}. Its use is to create a bridge between measures on the extended hull $\Omega_{X_0}^{ext}$ and measures on $G$. It does so through a simple formula:
 
 \begin{lemma}[e.g. \cite{machado2020apphigherrank}]\label{Lemma: Measure of pull-backs of small neighbourhoods through the periodization map}
  Let $X_0$ be a uniformly discrete subset of a locally compact second countable group $G$ and let $\nu$ be a Borel probability measure on $\Omega_{X_0}^{ext}$. Take an open subset $V \subset G$ and a compact subset $K \subset G$. Then 
  $$ \left(\mathcal{P}_{X_0}\right)^{*}\nu(V) = \int_{\Omega_{X_0}^{ext}} |X \cap V| d\nu(X) \leq |V^{-1}V\cap X_0^{-1}X_0|\nu(U^V),$$
  and 
  $$ \left(\mathcal{P}_{X_0}\right)^{*}\nu(K) = \int_{\Omega_{X_0}^{ext}} |X \cap K| d\nu(X) \leq |K^{-1}K\cap X_0^{-1}X_0|(1 - \nu(U_K)).$$
 \end{lemma}

  When applied to a proper $G$-invariant Borel probability measure, the pull-back by the periodization map is always a Haar measure. Similar observations were also exploited in \cite[\S 5]{bjorklund2016approximate} and \cite{10.2307/1969027}.
  
  \begin{corollary}\label{Corollary: Pull-back through periodization of G-invariant measure is a Haar-measure}
   Let $X_0$ be a uniformly discrete subset of a locally compact second countable group $G$ and let $\nu$ be a proper $G$-invariant Borel probability measure on $\Omega_{X_0}^{ext}$. Then $(\mathcal{P}_{X_0})^*\nu$ is a Haar-measure on $G$. 
  \end{corollary}

  We now state the central result of this section.

\begin{proposition}\label{Proposition: Reformulation approximate lattices}
Let $\Lambda$ be a uniformly discrete approximate subgroup in a unimodular second countable locally compact group $G$. There is some Borel subset $\mathcal{F} \subset G$ with finite (left-)Haar measure such that $\mathcal{F}\Lambda=G$ if and only if there exist a Borel probability measure $\nu$ on $\Omega_{\Lambda}$ and a constant $C \geq 1 $ such that $$ \frac{1}{C}\mu_G \leq \mathcal{P}_\Lambda^*\nu \leq C\mu_G$$
where $\mu_G$ is a left-Haar measure on $G$. 
\end{proposition}

\begin{proof}[Proof of the `if' part.]
    Let $\nu$ be as in the statement. Let $V$ be a compact neighbourhood of the identity such that $V^{-1}V \cap \Lambda^2 =\{e\}$ and let $(g_n)_{n \geq 0}$ be a sequence of elements of $G$ such that $G=\bigcup_{n \geq 0} g_nV$. Define inductively $B_n:=g_nV \setminus \left(\bigcup_{m<n} B_m\Lambda^2\right)$ and $B:=\bigcup_{n\geq 0}B_n$. Since $B_i\Lambda \cap B_j\Lambda = \emptyset$ for $i \neq j$ and the multiplication map $V \times \Lambda \rightarrow G$ is one-to-one, we have that the multiplication map $B \times \Lambda \rightarrow G$ is one-to-one. Moreover, we have that $gV_n \subset B_n \cup \bigcup_{m<n} B_m\Lambda^2$ for all $n \geq 0$. So $B\Lambda^2 = G$. Now the Haar measure $\mu_G$ is inner regular so there is a sequence of compact subsets $K_n \subset B$ with $\sup_{n \geq 0} \mu_G(K_n) = \mu_G(B)$. Since $K_n \subset B$, we have that $|K_n^{-1}K_n \cap \Lambda^2| \leq 1$. Hence, $\mathcal{P}_\Lambda^*\nu(K_n) \leq 1$ so $\mu_G(K_n) \leq C$ for all integers $n \geq 0$ by Lemma \ref{Lemma: Measure of pull-backs of small neighbourhoods through the periodization map}. So $\mu_G(B) \leq C < \infty$. 
\end{proof}

To prove the converse statement we first have to show a result about elements of the invariant hull of an approximate lattice.

\begin{lemma}\label{Lemma: Limits of approximate lattices have finite co-volume}
Let $\Lambda$ be a uniformly discrete approximate subgroup of a locally compact group $G$. Let $F$ be a finite subset such that $\Lambda^2 \subset \Lambda F$. If $\mathcal{F}\subset G$ is such that $\Lambda \mathcal{F} = G$, then for all $X \in \Omega_{\Lambda}\setminus\{\emptyset\}$, $XF\mathcal{F}=G$. 
\end{lemma}

\begin{proof}
Take $X \in \Omega_{\Lambda}\setminus\{\emptyset\}$. There is a sequence $(g_n)_{n \geq 0}$ such that $g_n\Lambda \rightarrow X$ as $n\rightarrow \infty$. Choose $x \in X$, then there is a sequence $(\lambda_n)_{n \geq 0}$ of elements of $\Lambda$ such that $g_n\lambda_n \rightarrow x$ i.e. $g_n=:x_n\lambda_n^{-1}$ with $x_n \rightarrow x$. But now 
$$g_n\Lambda F \supset g_n \Lambda^2 = x_n\lambda_n^{-1}\Lambda^2 \supset x_n\Lambda.$$ Taking limits on both sides we get $XF \supset x\Lambda$. Therefore, $XF\mathcal{F}\supset x\Lambda\mathcal{F}=G$. 
\end{proof}

\begin{proof}[Proof of Proposition \ref{Proposition: Reformulation approximate lattices}.]
We may assume that $\mathcal{F}\Lambda^3=G$ and the multiplication map $ \mathcal{F}\times\Lambda  \rightarrow G$ is one-to-one. Indeed, let $V$ be a neighbourhood of the identity such that the multiplication map $V \times \Lambda \rightarrow G$ is one-to-one and let $(g_n)_{n \geq 0}$ be a sequence of elements of $G$ such that $\mathcal{F} \subset \bigcup_{n \geq 0} g_nV$. Define inductively $\mathcal{F}_n:= \left(\mathcal{F} \cap g_nV \right)\setminus \left(\bigcup_{m < n} \mathcal{F}_m\Lambda^2\right)$. Clearly, for all $n$ and $m$ the multiplication map $\mathcal{F}_n \times \Lambda \rightarrow G$ is one-to-one and $\mathcal{F}_n \Lambda \cap \mathcal{F}_m \Lambda = \emptyset$. Define $\mathcal{F}_{\infty}:=\bigcup_{n \geq 0} \mathcal{F}_n$, we have that the multiplication map $\mathcal{F}_{\infty} \times \Lambda \rightarrow G$ is one-to-one. Moreover, we see inductively that $\mathcal{F} \cap g_n V \subset \mathcal{F}_n \cup \bigcup_{m < n} \mathcal{F}_m \Lambda^2$. So $\mathcal{F} \subset \mathcal{F}_{\infty} \Lambda^2$ and $G = \mathcal{F}_{\infty} \Lambda^3$. Upon replacing $\mathcal{F}$ with $\mathcal{F}_{\infty}$ our claim is proved.

Let $F \subset \Lambda^4$ be a finite subset such that $\Lambda^3 \subset F \Lambda$. Take also $F'$ such that $F\Lambda  \subset \Lambda^5 \subset \Lambda F'$. Now the map $ \mathcal{F}F  \times \Lambda\rightarrow G$ is surjective. We claim that the fibres of this map have size at most $|F||F'|$. Take $\lambda_1,\lambda_2 \in \Lambda$, $f_1,f_2 \in F$ and $v_1,v_2 \in \mathcal{F}$ such that $v_1f_1\lambda_1 =v_2f_2\lambda_2$. Choose $f_1',f_2' \in F'$ and $\lambda_1',\lambda_2' \in \Lambda$ such that $f_1\lambda_1=\lambda_1'f_1'$ and $f_2\lambda_2=\lambda_2'f_2'$. If $f_1'=f_2'$, we find that $v_1\lambda_1'=v_2\lambda_2'$. Because $\mathcal{F} \times \Lambda \rightarrow G$ is one-to-one, this means that $v_1=v_2$. In turn, we have $f_1\lambda_1=f_2\lambda_2$. So if $f_1=f_2$ we have $\lambda_1=\lambda_2$. In other words, given a choice of $f_1'$ and $f_1$, there is at most one solution. So the fibres are of size at most $|F||F'|$.

Write $\mathcal{F}_1=\mathcal{F}F$. For every continuous function with compact support $\phi \in C^0_c(G)$ define the quantity 
$$ M(\phi):=\int_{G} \mathds{1}_{\mathcal{F}_1}(g)\sum_{\lambda \in \Lambda}\phi(g\lambda) \mu_G(g)=\int_{G} \mathds{1}_{\mathcal{F}_1}(g)\mathcal{P}_{\Lambda}(\phi)(g\Lambda) \mu_G(g).$$
Since $G$ is unimodular and $\Lambda$ is symmetric we have
$$M(\phi)=\int_{G} \left(\sum_{\lambda \in \Lambda}\mathds{1}_{\mathcal{F}_1}(g\lambda)\right)\phi(g) \mu_G(g).$$ 
The second paragraph of the proof implies that for all $g \in G$, 
$$ 1 \leq\sum_{\lambda \in \Lambda}\mathds{1}_{\mathcal{F}_1}(g\lambda) \leq |F||F'|.$$
So if $\phi$ is non-negative, 
$$ \mu_G(\phi) \leq M(\phi) \leq |F||F'|\mu_G(\phi).$$

Therefore, $M$ defines a positive linear functional on $\mathcal{P}_{\Lambda}\left(C^0_c(G)\right)$ seen as a subspace of $C^0_c(\Omega_{\Lambda}\setminus\{\emptyset\})$. We wish now to apply the Hahn--Banach theorem - in the version of \cite[\S II.3, Proposition 1]{MR633754} - to extend $M$ as a positive linear functional of $C^0_c(\Omega_{\Lambda}\setminus\{\emptyset\})$. We show that this is possible thanks to Lemma \ref{Lemma: Limits of approximate lattices have finite co-volume}. Take $(V_n)_{n \geq 0}$ an increasing sequence of relatively compact open subsets of $G$ such that:
\begin{center}
$\mathcal{F}_1 \subset \bigcup_{n \geq 0} V_n $ and $\forall n \geq 0,\ \mu_G(V_n) \leq 2 \mu_G(\mathcal{F}_1)$.
\end{center}  
Consider now $(\phi_n)_{n \geq 0}$ a sequence of elements of $C^0_c(G)$ taking non-negative values such that for all $n$:
\begin{center}
 $\mu_G(\phi_n) \leq 2\mu_G(V_n)$ and $\forall g \in V_n,\ \phi_n(g)=1$.
\end{center}  
We have
 $$\mathcal{P}_{\Lambda}(\phi_n)(X) \geq |X \cap V_n|.$$ So  $\mathcal{P}_{\Lambda}(\phi_n)(X) \geq 1$ for all $X \in U^{V_n}$. According to Lemma \ref{Lemma: Limits of approximate lattices have finite co-volume},
$$\bigcup_{n \geq 0} U^{V_n} \supset \{ X \in \Omega_{\Lambda} : X \cap \mathcal{F}_1 \neq \emptyset\} = \Omega_{\Lambda}\setminus\{\emptyset\}.$$
Since the subsets $U^{V_n}$ are open, we have that for all $\psi \in C^0_c(\Omega_{\Lambda}\setminus\{\emptyset\})$ there are $n \geq 0$ and $c > 0$ such that $c\mathcal{P}(\phi_n) \geq \psi$. Therefore, we can apply the Hahn--Banach theorem \cite[\S II.3, Proposition 1]{MR633754}.

Hence, we find a positive linear functional $\nu$ on the set $C^0_c(\Omega_{\Lambda}\setminus\{\emptyset\})$ such that $\nu(\mathcal{P}_{\Lambda}(\phi))=M(\phi)$ for all $\phi \in C^0_c(G)$. According to the Riesz--Markov--Kakutani representation theorem, $\nu$ corresponds to integration against a Radon measure - that we denote by $\nu$ as well - on $\Omega_{\Lambda}\setminus \{\emptyset\}$. We extend $\nu$ to $\Omega_{\Lambda}^{ext}$ by assigning $\nu(\Omega_{\Lambda}^{ext}\setminus \Omega_{\Lambda})=0$. It remains only to prove that $\nu$ is a finite measure.  We have
 \begin{align*}
 \nu(\Omega_{\Lambda})=\nu\left(\bigcup_{n \geq 0} U^{V_n}\right) & \leq \liminf_n\nu ( \mathcal{P}_{\Lambda}(\phi_n))\\ 
 &\leq \liminf_n|F'||F|\mu_G(\phi_n) \\
 &\leq 4|F'||F|\mu_G(\mathcal{F}_1)<\infty.
 \end{align*}

\end{proof}

  We end this section with a proof of the unimodularity result mentioned earlier (Proposition \ref{Proposition: Envelope of a star-approximate lattice is unimodular}). We will require an additional result: 
  
  \begin{proposition}\label{Proposition: star-approximate lattices are bi-syndetic}
Let $\Lambda$ and $G$ be as in  the second part of Proposition \ref{Proposition: Reformulation approximate lattices} and let $V \subset G$ be any neighbourhood of the identity. Then there is a compact subset $K \subset G$ such that $V \Lambda K = G$.
\end{proposition}
  
  \begin{proof}
  We mimic the proof of \cite[Corollary 4.22]{bjorklund2016approximate}. Fix a measure $\nu$ on $\Omega_\Lambda^{ext}$ satisfying the conclusions of Proposition \ref{Proposition: Reformulation approximate lattices}. Recall that $\nu$ is a probability measure and $\nu(gU^V) \geq \frac{1}{C}\mu_G(V) > 0$ for all $g \in G$. So any subset $F \subset G$ such that the subsets $(fU^V)_{f\in F}$ are pairwise disjoint is finite. Take one such subset $F$ that is maximal for the inclusion. Then for all $g \in G$, there is $f \in F$ such that $gU^V \cap fU^V \neq \emptyset$. Take $X \in   gU^V \cap fU^V$. There is $x_1 \in X \cap gV$and $x_2 \in X \cap fV$. Hence, $$x_2^{-1}x_1 \in XX^{-1} \cap \left(V^{-1}f^{-1}gV\right) \subset \Lambda^2 \cap \left(V^{-1}f^{-1}gV\right).$$
Since $\Lambda^2 \cap V^{-1}f^{-1}gV$ is non-empty, $g \in FV\Lambda^2V$ which concludes. 
  \end{proof}
  
  \begin{proof}[Proof of Proposition \ref{Proposition: Envelope of a star-approximate lattice is unimodular}.]
  Fix a measure $\nu$ on $\Omega_\Lambda^{ext}$ satisfying the conclusions of Proposition \ref{Proposition: Reformulation approximate lattices} and let $\Delta_G$ be the modular function of $G$. Take $V$ a neighbourhood of the identity such that $\Lambda^4 \cap V^{-1}V=\{e\}$. For all $\lambda \in \Lambda$ we have $\Lambda^2 \cap \lambda^{-1}V^{-1}V\lambda=\{e\}$.  For all $\lambda \in \Lambda$ we have by Lemma \ref{Lemma: Measure of pull-backs of small neighbourhoods through the periodization map} applied to $\mu_G(V\lambda)$ that
  $$\Delta_G(\lambda)\mu_G(V) = \mu_G(V\lambda) \leq C.$$
  So $\Delta_G$ is bounded by $\mu_G(V)^{-1}$ on $\Lambda$. By Proposition \ref{Proposition: star-approximate lattices are bi-syndetic} this means that $\Delta_G$ is bounded. Hence, $G$ is unimodular. 
  \end{proof}

 \section{Quasi-models of uniformly discrete approximate subgroups}\label{Section: Quasi-models of uniformly discrete approximate subgroups}
 In this third intermediate section, our goal is to prove that BH-approximate lattices are always contained in approximate lattices. Incidentally, our method generalises Hrushovski's proof of \cite[Theorem 6.9]{hrushovski2020beyond}, see Corollary \ref{Corollary: minimal laminar supset is discrete in semi-simple groups}. This is not a surprise as we build upon the existence of quasi-models due to Hrushovski (Theorem \ref{Theorem: Hrushovski's quasi-models}) as well as his strategy behind the proof of \cite[Theorem 6.9]{hrushovski2020beyond}. One ingredient is the refined version of the quasi-model theorem for uniformly discrete approximate subgroups obtained below. 
 
 Fix an approximate subgroup $\Lambda$ of a locally compact group $G$ and write $\Gamma:=\Comm_G(\Lambda)$. Let $f: \Gamma_0 \rightarrow H$ be the quasi-model obtained by applying Theorem \ref{Theorem: Hrushovski's quasi-models} to $\Lambda$ and $\Gamma$ with $A\simeq \mathbb{R}^n$ and $f(\Gamma)$ relatively dense (Lemma \ref{Lemma: Preparing quasi-models}). Recall that $K$ denotes a Euclidean ball in $A$ containing the defect of $f$, that we assume that it is a Euclidean ball contained in $A$ and that $\Gamma_0$ has finite index in $\Gamma$. We also assume that $\Gamma=\Gamma_0$ for convenience, as this does not affect the proof of Theorem \ref{Proposition: Detail resultat principal}. 
 
  Our main focus is on the study of the approximate subgroup
 $$\Gamma_{f,K}:=\{(\gamma, f(\gamma)k) \in G \times H : \gamma \in \Gamma ,k \in K \}$$
 and its properties.

 \begin{lemma}\label{Lemma: Refinement quasi-model, uniformly discrete case}
 Let $\Lambda, \Gamma, G, H, A$ and $f$ be as above. There are $A_1, A_2$ vector subspaces of $A$ such that for all symmetric relatively compact neighbourhoods of the identity $W_G \subset G, W_H \subset H$, we have $\left(W_G \times A_1W_H \right)\cap \Gamma_{f,K}^2$ relatively dense in  $W_G \times A_1W_H$ and $\left(W_G \times A_2W_H \right)\cap \Gamma_{f,K}^2$ relatively compact.
 \end{lemma}

 \begin{proof}
 The subset $\Xi:=\left(W_G \times AW_H \right)\cap \Gamma_{f,K}^2$ is an approximate subgroup (Lemma \ref{Lemma: Intersection of approximate subgroups}). According to Lemma \ref{Lemma: Schreiber around subspace} there are a vector subspace $A_1 \subset A$ and a compact subset $K' \subset G \times H $ such that $\Xi \subset K'A_1$ and $A_1 \subset K'\Xi$. Moreover, since $\Gamma_{f,K}$ commensurates $\Xi$, $\Gamma_{f,K}$ normalises $A_1$.  As the projection of $\Gamma_{f,K}$ to $H/A$ is dense (it is equal to the projection of $f(\Gamma)$), $A_1$ is normal. Recall that the action by conjugation on $A$ is by isometries, so the orthogonal subspace $A_2$ of $A_1$ in $A$ is normal too. We have also that 
 $$\left(W_G \times A_2W_H \right)\cap \Gamma_{f,K}^2 \subset (W_G \times A_2W_H) \cap K'A_1$$
 and the latter is relatively compact since $A_1$ and $A_2$ are orthogonal. So Lemma \ref{Lemma: Refinement quasi-model, uniformly discrete case} is true. 
\end{proof}

Our next lemma explores what happens once we have quotiented $A_2$ out. 

\begin{lemma}\label{Lemma: construction lambda star of a BH-app lattice}
 With $A_1$ and $A_2$ as in Lemma \ref{Lemma: Refinement quasi-model, uniformly discrete case} applied to $\Lambda$ and $f$. Define $f^*:\Gamma \rightarrow H^*:=H/A_2$ as the composition of $f$ and the natural projection $H \rightarrow H/A_2$. Write $K^*$ and $A^*$ the projections of $K$ and $A$ to $H^*$ respectively.  There is a neighbourhood of the identity $W^* \subset H^*$ such that $\Lambda^*:=f^{-1}(W^*K^*)$ is an approximate subgroup, $\Lambda $ is covered by finitely many translates of $\Lambda^*$, $\Lambda^*$ is uniformly discrete in $G$ and $\Gamma_{f^*,K^*}$ is relatively dense in the subgroup $\overline{\Gamma_{f^*,K^*}A^*}$.
 \end{lemma}
 
 \begin{proof}
We have that $\Gamma_{f^*,K^*}$ is an approximate subgroup of $G \times H^*$, and for any relatively compact neighbourhood of the identity $W_G \subset G$ and $W_H \subset H$ the subset $\Gamma_{f,K}^2 \cap \left(W_G \times W_HA_2\right)$ is relatively compact. In addition, 
$$(f^*)^{-1}(W_H^*) = f^{-1}(W_HA_2)$$
where $W_H^*$ denotes the projection of $W_H$ to $H^*$. Thus,  $$(f^*)^{-1}(W_H^*)^2 \cap W_G =f^{-1}(W_HA_2)^2 \cap W_G \subset f^{-1}(W_H')^2 \cap W_G$$ for some potentially larger compact subset $W_H' \subset H$. But $f^{-1}(W_H')$ is contained in a uniformly discrete approximate subgroup since it is covered by finitely many translates of $\Lambda$, so $f^{-1}(W_HA_2)^2 \cap W_G$ is finite. In turn, $(f^*)^{-1}(W_H^*)=f^{-1}(W_HA_2)$ is contained in a uniformly discrete approximate subgroup. This shows that for any choice of symmetric relatively compact neighbourhood of the identity $W^* \subset A^*$, the subset $\Lambda^*:= (f^*)^{-1}(W^*K^*)$ is a uniformly discrete approximate subgroup. Since $\Gamma_{f,K}^2 \cap \left(W_G \times W_HA_1\right)$ is relatively dense in $W_G \times W_HA_1$, and $A_1$ projects surjectively onto $A/A_2$, we have that $\Gamma_{f^*,K^*}^2 \cap \left(W_G \times W_H^*A^*\right)$ is relatively dense in $W_G \times W_H^*A^*$. So there is a compact neighbourhood $W$ of the identity in $G \times H^*$ such that $A^* \subset \Gamma_{f^*,K^*}W$. Hence, $\Gamma_{f^*,K^*}A^* \subset \Gamma_{f^*,K^*}^2W$. As $\Gamma_{f^*,K^*}$ is an approximate subgroup, $\Gamma_{f^*,K^*}$ is relatively dense in $\overline{\Gamma_{f^*,K^*}A^*}$.
 \end{proof}

In order to exploit the construction of $\Lambda^*$, we rely on the following elaboration on the proof of unimodularity of the envelope of an approximate lattice, see \S \ref{Proposition: Envelope of approximate lattice is unimodular}.

  \begin{lemma}\label{Lemma: Elaboration proof unimodularity}
  With notation as in Lemma \ref{Lemma: construction lambda star of a BH-app lattice} and suppose for simplicity that $\Lambda = \Lambda^*$. Then:
  \begin{enumerate}
  \item if $L \subset G \times H$ is a closed subgroup containing $A$ and $B \subset L$ is a Borel subset such that $\Gamma_{f,K}B=L$, then for any Borel subset $B' \subset L$,
  $$\mu_L(B') \leq \mu_L(KB) |f^{-1}(p_H(B'B'^{-1}K)) \cap p_G(B'B'^{-1})| $$
  where $p_G: G \times H \rightarrow G$ and $p_H: G \times H \rightarrow H$ are the natural projections and $\mu_L$ denotes a Haar measure on $L$;
  \item the subgroup $L_0:= \overline{\Gamma_{f,K}A}$ is unimodular. 
  \end{enumerate}
 \end{lemma}
 
 \begin{proof}
 Choose $X \subset \Gamma_f:=\{(\gamma, f(\gamma)) : \gamma \in \langle \Lambda \rangle\}$ such that $B' \subset \bigcup_{\gamma \in X} \gamma KB$. By counting redundancies, we have 
 \begin{align*}
 \mu_L(B') &\leq \mu_L\left(\bigcup_{\gamma \in X} KB \cap \gamma^{-1}B'\right)r \\
 &\leq  \mu_L(KB)r
 \end{align*}
 where $r$ denotes the maximal number of distinct $\gamma_1, \ldots, \gamma_s \in X$ such that $\gamma_1^{-1}B' \cap \ldots \cap \gamma_r^{-1}B' \neq \emptyset$.
 Notice that $\gamma_{1}^{-1}B' \cap \ldots \cap \gamma_{r}^{-1}B' \neq \emptyset$ implies that $X'X'^{-1} \subset B'B'^{-1}$ where $X':=\{\gamma_1, \ldots, \gamma_r\}$. Since $\gamma_1, \ldots , \gamma_r \in \Gamma_f$ and $f$ is a quasi-homomorphism with defect $K$, we have the inclusion $$p_G(X'X'^{-1}) \subset f^{-1}(p_H(B'B'^{-1}K)) \cap p_G(B'B'^{-1}).$$ But the projection of $\Gamma_f$ to $G$ is injective. So $r$ is at most $|f^{-1}(p_H(B'B'^{-1}K)) \cap p_G(B'B'^{-1})|$. This proves (1). 

By Lemma \ref{Lemma: construction lambda star of a BH-app lattice} (recall that we assume $\Lambda=\Lambda^*$), there is a compact subset $W$ of $L_0$ such that $\Gamma_{f,K}W=L_0$. Write $s:=|f^{-1}(p_H(WW^{-1}K)) \cap p_G(WW^{-1})|$.  The quantity $s$ is finite since $f^{-1}(p_H(WW^{-1}K))$ is uniformly discrete and $W$ is compact. Take $g \in L_0$. By applying (1) with $L:=L_0$, $B:=W$ and $B':=Wg$ we have 
$$\mu_{L_0}(Wg) \leq \mu_{L_0}(KW)s.$$
In other words, the modular function of $L_0$ is bounded above i.e. is trivial. This concludes the proof of (3). 
\end{proof}
 
 The second ingredient is a coupling argument developed in \cite{bjorklund2019borel}. Given two closed subsets $X,Y$ of a locally compact group $G$, one can define the topological coupling 
 $$\Omega_{X,Y}:=\overline{\{(gX,gY) : g \in G\}} \subset \mathcal{C}(G) \times \mathcal{C}(G)$$ as the closure  of the orbit of $(X,Y)$ in $\mathcal{C}(G) \times \mathcal{C}(G)$ under the diagonal $G$-action. The system $\Omega_{X,Y}$ comes naturally with continuous $G$-equivariant projections onto $\Omega_X$ and $\Omega_Y$ and can be used to transfer some measure-theoretic properties from $\Omega_X$ to $\Omega_Y$ and conversely.
 
 \begin{lemma}\label{Lemma: Stationary measures and relative density}
 Let $X$ be a closed subset of a locally compact group $G$. Let $\mu$ be a Borel probability measure on $G$ and $Y \subset G$ be another closed subset such that there exists a compact subset $C \subset G$ with $X \subset YC$. If there exists a proper $\mu$-stationary probability measure on $\Omega_X$, then there exists a proper $\mu$-stationary probability measure on $\Omega_Y$. 
 \end{lemma}
 
 \begin{proof}
 First of all, we claim that if $(X',Y') $ belongs to $\Omega_{X,Y}$ then $X' \subset Y'C$. Indeed, by construction there is a sequence $(g_n)_{n \geq 0}$ such that $g_nX \rightarrow X'$ and $g_n Y \rightarrow Y'$ as $n$ goes to infinity. But $g_nX \subset g_nYC$. Hence, if $x \in X'$ we can find sequences $(y_n)_{n \geq 0}$ and $(c_n)_{n \geq 0}$ of elements of $Y$ and $C$ respectively such that $g_ny_nc_n \rightarrow x$ as $n$ goes to infinity. Upon considering a subsequence, we may assume that $(c_n)_{n \geq 0}$ converges to $c \in C$. So $g_ny_n \rightarrow xc^{-1}$ as $n$ goes to infinity. In turn, this means that $xc^{-1} \in Y'$ which establishes the claim. 
 
 Fix now a proper $\mu$-stationary probability measure $\nu_X$ on $\Omega_X$. Let $p_X: \Omega_{X,Y} \rightarrow \Omega_X$ and  $p_Y: \Omega_{X,Y} \rightarrow \Omega_Y$ denote the natural projections. By a Hahn--Banach argument, there is a $\mu$-stationary Borel probability measure $\nu_{X,Y}$ on $\Omega_{X,Y}$ such that $\pi_X^*\nu_{X,Y}=\nu_X$, see \cite[Appendix 1]{bjorklund2019borel}. Now, $\nu_Y:=\pi_Y^*\nu_{X,Y}$ is a $\mu$-stationary Borel probability measure on $\Omega_Y$ and we claim it is proper. Indeed, suppose $\nu_Y(\{\emptyset\}) > 0$, then $\nu_{X,Y} ( \pi_Y^{-1}(\{\emptyset\})) > 0$. But $(X', \emptyset) \in \Omega_{X,Y}$ implies $$X' \subset \emptyset \cdot C = \emptyset$$ according to the first paragraph. So $\nu_{X,Y}(\{(\emptyset, \emptyset)\}) > 0$ which yields $\nu_X(\{\emptyset\}) > 0$. A contradiction. 
 \end{proof}
 
 \begin{proof}[Proof of Theorem \ref{Proposition: Detail resultat principal}(4).]
Let $\Lambda^*$ be given by Lemma \ref{Lemma: construction lambda star of a BH-app lattice}. Since $\Lambda^*$ is a uniformly discrete approximate subgroup finitely many translates of which cover $\Lambda$, $\Lambda^*$ is a BH-approximate lattice as well (Lemma \ref{Lemma: Stationary measures and relative density} or \cite{bjorklund2019borel}). To prove the remaining parts of Theorem \ref{Proposition: Detail resultat principal}, we may assume that we had started with $\Lambda = \Lambda^*$ and $f=f^*$. Define $L_0:=\overline{\Gamma_{f,K}A}$. Fix an admissible probability measure $\mu$ on $G$ (Definition \ref{Definition: BH-approximate lattice}) and a proper $\mu$-stationary probability measure $\nu$ on $\Omega_{\Lambda}$. Since $\{(\lambda, e) \in G \times H : \lambda \in \Lambda\}$ is contained in $L_0 (\{e\} \times W)$ for some compact subset $W$, Lemma \ref{Lemma: construction lambda star of a BH-app lattice} shows that there is a $\mu$-stationary probability measure on $\left(G \times H\right)/L_0$. Since the projection of $L_0$ to $H$ is dense, $\left(\overline{\Gamma} \times H\right) / L_0$ admits a Haar measure of finite total mass where we look at $\Gamma$ as a subgroup of $G \times \{e\}$,  see Lemma \ref{Lemma: Technical bit invariant measures}. As $L_0$ is unimodular (Lemma \ref{Lemma: Elaboration proof unimodularity}), we have that $\overline{\Gamma} \times H$ is unimodular \cite[\S I]{raghunathan1972discrete}.

We claim that $\Lambda$ is an approximate lattice in $\overline{\Gamma}$. Since the projection of $L_0$ to $\overline{\Gamma}$ is dense, the action of $\overline{\Gamma}$ on $\left(\overline{\Gamma} \times H\right) / L_0$ is minimal (e.g. \cite[2.2.3]{zimmer2013ergodic}). So take any relatively compact neighbourhood of the identity $W \subset \overline{\Gamma} \times H$. Then $\Gamma W$ projects surjectively to $\left(\overline{\Gamma} \times H \right)/ L_0$. One can therefore build a Borel subset $B \subset \Gamma W$ whose projection to $\left(\overline{\Gamma} \times H \right)/ L_0$ is one-to-one and surjective onto a set of full measure (using \cite{MR7916}). By construction, the projection of $B$ to $H$ is relatively compact. Now, Lemma \ref{Lemma: construction lambda star of a BH-app lattice} tells us that there exists a relatively compact neighbourhood of the identity $W \subset L_0$ such that $W\Gamma_{f,K}= L_0$. So $\mathcal{F}:=BW$ satisfies $\mathcal{F} \Gamma_{f,K} = \overline{\Gamma} \times H$ and $$\mu_{\overline{\Gamma} \times H}(\mathcal{F}K)=\mu_{\overline{\Gamma} \times H/L_0}(\overline{\Gamma} \times H/L_0)\mu_{L_0}(WK) <\infty.$$  The above equality is a consequence of the usual quotient formula for unimodular subgroups \cite[\S 1]{raghunathan1972discrete}. Let $W_H \subset H$ denote a relatively compact neighbourhood of the identity such that $\Lambda \supset f^{-1}(W_HK)$ and choose another neighbourhood of the identity $W_H'$ such that $W_H'W_H'^{-1} \subset W_H$. Consider also any Borel subset $B \subset \overline{\Gamma}$ such that $BB^{-1} \cap \Lambda =\{e\}$. We want to show that $B$ has finite Haar measure and then apply the simple criterion \cite[App. 1]{hrushovski2020beyond} to prove that $\Lambda$ is an approximate lattice. But $\mu_{\overline{\Gamma} \times H}(B \times W_H')$ is bounded above by $\mu_{\overline{\Gamma} \times H}(\mathcal{F}K)|\Lambda \cap BB^{-1}|= \mu_{\overline{\Gamma} \times H}(\mathcal{F}K)$ (Lemma \ref{Lemma: Elaboration proof unimodularity}) which is finite. So the proof is complete. 
 \end{proof}
 
 It only remains to prove the technical result: 
 
 \begin{lemma}\label{Lemma: Technical bit invariant measures}
 Let $\Gamma$ be a closed subgroup of  a product of two locally compact groups $G \times H$. Suppose that $\mu$ is an admissible Borel probability measure on $G$ and $\nu$ is a $\mu$-stationary Borel probability measure on $\left(G \times H\right)/\Gamma$. Suppose that the projection of $\Gamma$ to $H$ is dense. Then there is a Borel probability measure on $G \times H$ invariant under $\Gamma$ and $H$. 
 \end{lemma}
 
 \begin{proof}
 For any admissible probability measures $\mu_1, \mu_2$ on $H$, define $\nu_1=\mu_1 * \nu$ and $\nu_2 = \mu_2 * \nu$. Both $\nu_1$ and $\nu_2$ are $\mu$-stationary and in the Haar measure class. Let $\phi$ denote the Radon-Nykodym derivative of $\nu_2$ against $\nu_1$. An easy computation provides 
 $$\int_{G \times G} |\phi(g) - \phi(hg)|^2d\mu(h)d\nu_1(g) = 0$$
 i.e. $\phi$ is invariant under the action of $G$.  But $G$ acts ergodically on $G \times H/\Gamma$ according to \cite[2.2.3]{zimmer2013ergodic}. Thus, $\nu_1 = \nu_2$. Since this is true for all $\mu_1$ and $\mu_2$ we find by an approximation argument $\nu_1 = \nu$ and $\nu$ is invariant under translation by elements of $H$. Take a measure $\nu'$ appearing in the $H$-ergodic decomposition of $\nu$. Then $\nu'$ is $H$-invariant and up to translation by an element of $G$, we can assume that its support is contained in the orbit closure of $\Gamma$ (seen as a coset). Now, we see - by repeating the above idea with $G$ playing the role of $H$ for instance - that $\nu'$ is also invariant under $\overline{p_G(\Gamma)}$ where $p_G: G \times H \rightarrow G$ is the natural projection. This yields the desired result.  
 \end{proof}
 
 Pushing these ideas further one could obtain a generalisation of the Meyer-type theorem due to Hrushovski regarding approximate lattices of abstractly semi-simple locally compact groups. 
 
\begin{corollary}\label{Corollary: minimal laminar supset is discrete in semi-simple groups}
Let $\Lambda$ be a uniformly discrete approximate subgroup contained in an abstractly semi-simple locally compact group $G$ and such that $\Comm_G(\Lambda)$ is dense in $G$. Then there is a uniformly discrete \emph{laminar} approximate subgroup $\Lambda^*$ finitely many translates of which cover $\Lambda$.
\end{corollary}
 
 When $\Lambda$ is an approximate lattice, then $\Lambda$ and $\Lambda^*$ are commensurable by Lemma \ref{Corollary: Towers of star-approximate subsets are commensurable}. So $\Lambda$ is already laminar, in effect generalising \cite[Thm. 7.4]{hrushovski2020beyond}.

\section{A laminar approximate lattice that is not strong}\label{Subsection: Consequences of the existence of non-laminar approximate lattices}
The proof of Theorem \ref{Proposition: Detail resultat principal} (4) is now an easy consequence of the laminarity of $\star$-approximate lattices. 

\begin{proof}[Proof of Theorem \ref{Proposition: Detail resultat principal} (4)]
According to \cite[\S 7.9]{hrushovski2020beyond}, there exists an approximate lattice $\Lambda$ that is not laminar. According to Theorem \ref{Proposition: Detail resultat principal} (1), $\Lambda$ is not commensurable with a $\star$-approximate lattice. 
\end{proof}

We turn now to the proof of Proposition \ref{Proposition: laminar app lattice not close to a strong app lattice} which will be more involved. This will be a consequence of Theorem \ref{Proposition: Detail resultat principal} and the following:

\begin{proposition}\label{Proposition: presque la fin}
Let $\Gamma$ be a lattice in a rank one simple Lie group $G$ and $m \geq 0$ be an integer. There is an approximate lattice $\Lambda \subset \Gamma$ such that $\Lambda^m$ does not contain a model set. 
\end{proposition}

Since $\Lambda$ is an approximate lattice in $\Gamma$, it is commensurable with $\Gamma$. Therefore, Proposition \ref{Proposition: presque la fin} provides laminar approximate lattices that are arbitrarily far from model sets. This is in stark contrast with the methods coming from additive combinatorics which usually yield that a fixed power - usually $4$ - is already a model set. We invite the interested reader to compare with Theorem \ref{Proposition: Detail resultat principal} (1) and \cite{machado2019goodmodels}. 

In order to prove Proposition \ref{Proposition: presque la fin} we proceed as in the proof of Theorem \ref{Proposition: Detail resultat principal} (4) and start by considering an non-laminar approximate lattice. We take advantage of the following construction:

\begin{lemma}\label{Lemma: Non-laminar approximate lattice}
Let $\Gamma$ be a lattice in a rank one simple Lie group $G$. There is a non-laminar approximate lattice $\Lambda \subset G \times \mathbb{R}$ that projects surjectively to $\Gamma \subset G$.
\end{lemma}

\begin{proof}Let $f: \Gamma \rightarrow \mathbb{R}$ be a quasi-morphism that is not within bounded distance of a group homomorphism \cite{BestvinaBrombergFujiwara2016Bounded} with defect smaller than $1$. Since $\mathbb{Z}[\sqrt{2}]$ is dense in $\mathbb{R}$, we may modify $f$ by a uniformly bounded (arbitrarily small) amount and assume that $f$ takes values in $\mathbb{Z}[\sqrt{2}]$. Write also $\tau: \mathbb{Z}[\sqrt{2}] \rightarrow \mathbb{Z}[\sqrt{2}]$ the Galois conjugation. Finally, suppose as we may that $f$ is symmetric. Define then the approximate lattice $\Lambda$ in $G \times \mathbb{R}$ by
$$\Lambda:=\{(\gamma, x) \in \Gamma \times \mathbb{Z}[\sqrt{2}]: |f(\gamma) - \tau(x)| \leq 1\}.$$
The subset $\Lambda$ is easily seen to be an approximate lattice in $G \times \mathbb{R}$. Moreover, the map $(\gamma, x) \mapsto (f(\gamma),\tau(x))$ is a quasi-morphism not within bounded distance of a group homomorphism. So $\Lambda$ is not laminar by e.g. \cite{machado2019goodmodels}.

\end{proof}

\begin{proof}[Proof of Proposition \ref{Proposition: presque la fin}.]
Choose an approximate lattice $\Lambda \subset G \times \mathbb{R}$ given by Lemma \ref{Lemma: Non-laminar approximate lattice} and fix $K \geq 0$ such that $\Lambda$ is a $K$-approximate subgroup. Suppose that there is an $m \geq 0$ such that there is no approximate lattice $\Lambda' \subset \Gamma$ in $G$ such that $\Lambda'^m$ does not contain a model set. Write $B_n:=[-n;n]$ and note that $B_n$ is a $2$-approximate subgroup. Let $p_G: G \times \mathbb{R} \rightarrow G$ denote the natural projection. Let $\Lambda_G$ denote $p_G(\Lambda^2 \cap G \times [-\epsilon; \epsilon])$ for some $\epsilon > 0$. We know that $\Lambda_G$ is laminar since it is commensurable with $\Gamma$, so $\Lambda_G^m$ is a model set by our assumption. We first show that $\Lambda^2 \cap (G \times [-\epsilon; \epsilon])$ has a good model too for a choice of $\epsilon$ sufficiently small. 

To do so, we build a local$\setminus$Freiman homomorphism from $\Lambda_G^{3m}$ to $\Lambda^2 \cap (G \times [-\epsilon; \epsilon])$.
Note that $\Lambda^{3m}$ is a uniformly discrete approximate subgroup. So there is $\epsilon_0 > 0$ such that $\Lambda^{9m} \cap \{e\} \times [-3\epsilon_0; 3\epsilon_0]=\{e\}$. Now, for any $\lambda \in p_G(\Lambda^{3m} \cap G \times [-\epsilon_0; \epsilon_0])$ let $\tau(\lambda)$ denote any element of $\Lambda^{3m} \cap G \times [-\epsilon_0; \epsilon_0]$ such that $p_G(\tau(\lambda))=\lambda$.  If $\lambda_1, \lambda_2, \lambda_1\lambda_2 \in p_G(\Lambda^{3m} \cap G \times [-\epsilon_0; \epsilon_0])$, then $$\tau(\lambda_1)\tau(\lambda_2)\tau(\lambda_1\lambda_2)^{-1} \in \left(\{e\} \times [-3\epsilon_0; 3\epsilon_0]\right) \cap \Lambda^{9m} =\{e\}.$$ So, if we take $\epsilon = \min(\epsilon_0/2m, 1)$, $\tau(\lambda_1)\tau(\lambda_2)=\tau(\lambda_1\lambda_2)$ for all $\lambda_1, \lambda_2 \in \Lambda_G$. By our assumption, $\Lambda_G^m$ has a good model. So $\tau(\Lambda_G^m)$ has a good model by Lemma \ref{Lemma: Charac. good models}. Upon replacing $\Lambda$ with $\Lambda^m$ at the beginning, we may therefore assume that $\Lambda^2 \cap G \times [-\epsilon; \epsilon]$ has a good model. For every $n > 0$ now, the $(2K)^3$-approximate subgroup $\Lambda_n=\Lambda^2 \cap G \times B_n$ is commensurable with $\Lambda^2 \cap G \times [-\epsilon; \epsilon]$  (Lemma \ref{Lemma: Intersection of approximate subgroups}) and contains $\Lambda^2 \cap G \times [-\epsilon; \epsilon]$. By Lemma \ref{Lemma: Charac. good models} again, this approximate subgroup must have a good model. Thus, according to \cite[Prop. 3.7]{machado2019goodmodels}, the approximate subgroup $\Lambda^4=\bigcup_{i \in I} \Lambda_i^2$ is laminar. 
\end{proof}

\begin{proof}[Proof of Proposition \ref{Proposition: laminar app lattice not close to a strong app lattice}.]
Choose $m >0$. According to Proposition \ref{Proposition: presque la fin}, there is an approximate lattice $\Lambda \subset G$ such that $\Lambda^{8m}$ does not contain a model set. By Theorem \ref{Proposition: Detail resultat principal} (1), $\Lambda^m$ cannot contain a $\star$-approximate lattice. This concludes the proof. 
\end{proof}

\section{Wrapping up the proof}\label{Subsection: Wrapping up the proof}
The purpose of this last section is to finally establish Figure \ref{Figure: Hierarchy of approximate lattices} and Theorem \ref{Proposition: Detail resultat principal}. We would like to draw the reader's attention to the fact that some of the results presented as part of Figure \ref{Figure: Hierarchy of approximate lattices} were already known before. Here is a list of remarks in this direction: 

\begin{enumerate}
\item It was shown in \cite[Rem. 4.14.(1)]{bjorklund2016approximate} that when the ambient group is amenable BH-approximate lattices are strong approximate lattices, in \cite{machado2019goodmodels} we proved that BH-approximate lattices are laminar and that they are often uniform. In short, in amenable groups, all notions of approximate lattices are (almost) equivalent; 
\item In \cite{bjorklund2016approximate} they showed that uniform approximate lattices are BH-approximate lattices; 
\item Model sets with regular windows were known to be strong approximate lattices by \cite{bjorklund2016aperiodic} and model sets with general windows were known to be $\star$-approximate lattices by \cite{machado2020apphigherrank};
\item Non-uniform approximate lattices are plenty, as should already be clear, see for instance \cite[\S 7]{hrushovski2020beyond} and \cite[\S 2.1.1]{machado2020apphigherrank}  for constructions with a number-theoretic flavour;
\item for all $m \geq 0$, there is an approximate lattice $\Lambda$ of some locally compact second countable group such that $\Lambda^m$ is \emph{not} a $\star$-approximate lattice, see \S \ref{Subsection: Consequences of the existence of non-laminar approximate lattices}. Moreover, $\Lambda$ can be chosen uniform \emph{and} laminar; 
\item for all $m \geq 0$, there is a laminar approximate lattice $\Lambda$ of some locally compact second countable group such that $\Lambda^m$ does not contain a model set, see \S \ref{Subsection: Consequences of the existence of non-laminar approximate lattices};
\item there remain only one outstanding questions: 
\begin{question}
Can one prove that BH-approximate lattices are approximate lattices?
\end{question}
\end{enumerate}

We can now prove Figure \ref{Figure: Hierarchy of approximate lattices}.

\begin{proof}[Proof of Figure \ref{Figure: Hierarchy of approximate lattices} and Theorem \ref{Proposition: Detail resultat principal}.]
We will prove the implications roughly starting from the left of the diagram and going to the right.  As mentioned above model sets with regular windows are strong approximate lattices by \cite{bjorklund2016aperiodic}. Now, the fact that a strong approximate lattice is a $\star$-approximate lattice is a consequence of Lemma \ref{Lemma: Extended invariant hull contains invariant hull}.  Conversely, if $\Lambda$ is a $\star$-approximate lattice, $\Lambda^8$ has a good model by Theorem \ref{Proposition: Detail resultat principal} and, hence,  contains a model set by Lemma \ref{Lemma: good models and model sets}. That $\star$-approximate lattices are approximate lattices is a consequence of Proposition \ref{Proposition: Reformulation approximate lattices} combined with the unimodularity of their envelope (Proposition \ref{Proposition: Envelope of a star-approximate lattice is unimodular}). It is also clear that uniform approximate lattices are types of approximate lattices. It remains to prove that approximate lattices are BH-approximate lattices and its partial converse. Take $\Lambda$ an approximate lattice in a  second countable locally compact group $G$ and remark that $G$ must be unimodular (Lemma \ref{Proposition: Envelope of approximate lattice is unimodular}), $\mu$ a Borel probability measure on $G$, and $\nu$, $C$ and $\mu_G$ as in Proposition \ref{Proposition: Reformulation approximate lattices}. Then for all $\phi \in C^0_c(G)$ taking non-negative values

$$\mu * \nu(\mathcal{P}_{\Lambda}(\phi))   = \int_G\nu(\mathcal{P}_{\Lambda}(\phi_g)) d\mu(g) $$
where $\phi_g$ denotes the map $h \in G \mapsto \phi(gh)$.
Therefore, using the bounds from Proposition \ref{Proposition: Reformulation approximate lattices} and left-invariance of the Haar measure, we have
$$ \frac{1}{C}\mu_G(\phi) \leq \mu * \nu(\mathcal{P}_{\Lambda}(\phi)) \leq C \mu_G(\phi).$$
Since this is true for any such $\mu$, we find that any cluster point $\nu'$ of the sequence $\left(n^{-1}\sum_{i=1}^n\mu^{*i}*\nu \right)_{n \geq 0}$ is $\mu$-stationary and satisfies 
$$ \frac{1}{C}\mu_G(\phi) \leq \nu'(\mathcal{P}_{\Lambda}(\phi)) \leq C \mu_G(\phi).$$
So $\nu'$ is not concentrated on $\emptyset$. So $\Lambda$ is a BH-approximate lattice. Finally, the partial converse is Theorem \ref{Proposition: Detail resultat principal}(4). The proof of Figure \ref{Figure: Hierarchy of approximate lattices} is complete.
\end{proof}

\begin{appendix}
\section{A variation around Schreiber's theorem}

We now prove the technical lemma (Lemma \ref{Lemma: Schreiber around subspace})  used earlier in the proof of Lemma \ref{Lemma: Preparing quasi-models}. It will be deduced easily from a combination of a generalisation of Schreiber's theorem to the solvable setup \cite{machado2019infinite} and results concerning approximate subgroups contained in neighbourhoods of normal amenable subgroups \cite{machado2019goodmodels}.

\begin{proof}[Proof of Lemma \ref{Lemma: Schreiber around subspace}.]
Assume we are given a subspace $V \subset A$ (recall that $A$ is normal closed and isomorphic to a Euclidean space) and a compact subset $K' \subset G$ such that $\Lambda \subset VK'$ and $V \subset \Lambda K'$. If $\gamma \in \Comm_{G}(\Lambda)$ and $F$ denotes a finite subset such that $\gamma \Lambda \gamma^{-1} \subset \Lambda F$, then 
$$\gamma V \gamma^{-1} \subset \gamma \Lambda \gamma^{-1} \gamma K'\gamma^{-1} \subset \Lambda F\gamma K'\gamma^{-1} \subset VK'F\gamma K'\gamma^{-1}.$$
So $\gamma V \gamma^{-1}  = V$ i.e. $\Comm_{G}(\Lambda)$ normalises $V$. Thus, $\Lambda \subset V(K' \cap N_G(V))$ and $V \subset \Lambda (N_G(V) \cap K')$ by the modular law. To summarise, we only have to prove the existence of $V$ and the other conclusions follow.

Suppose now that $G$ is a connected Lie group. Choose a symmetric neighbourhood of the identity $W \subset G$ such that $W^8A$ does not contain any closed subgroup not contained in $A$ (e.g. by choosing a sufficiently small neighbourhood of the identity in $G/A$). Define $\Lambda':= \overline{W^2A \cap \Lambda^2}$. Since finitely many translates of $W$ cover $K$, finitely many translates of $WA$ cover $KA$ and $\Lambda$. By Lemma \ref{Lemma: Intersection of commensurable sets}, $\Lambda'$ is commensurable with $\Lambda$. According to \cite[Prop. 5.8]{machado2019goodmodels}, $\Lambda'$ is an amenable closed approximate subgroup of $G$. By the structure of amenable approximate subgroups \cite[Thm 1.6]{machado2019goodmodels}, there are a compact approximate subgroup $C \subset \Lambda'^4$, a closed approximate subgroup $\Lambda_{sol} \subset \Lambda'^4$ and a closed subgroup $N \subset \Lambda_{sol}$ normalised by $\Lambda_{sol}$ such that $\Lambda' \subset C \Lambda_{sol}$ and $\langle \Lambda_{sol}\rangle/N$ is solvable. By our choice of $W$, $N \subset A$ so $\langle \Lambda_{sol} \rangle$ is solvable. Write $L$ the closure of $\langle \Lambda_{sol} \rangle A$. Then $L$ is a solvable Lie group, $A \subset L$ and it is clearly enough to prove the result for $\Lambda_{sol} \subset L$.

By the Ado--Iwasawa theorem \cite[Proof of XVIII.3.2]{zbMATH03212078}, there is a faithful representation $\rho: L \rightarrow \GL_n(\mathbb{R})$ such that $\rho(A)$ is unipotent, hence closed.  So the map $\rho_{VA}$ is proper for any symmetric relatively compact neighbourhood of the identity $V$ in $L$. According to \cite{machado2019infinite}, there is a closed connected subgroup $N$ of $\GL_n(\mathbb{R})$ normalised by $\rho(L)$ and a compact subset $C$ such that $\rho(\Lambda_{sol}) \subset CN$ and $N \subset C\rho(\Lambda_{sol})$. We want to be able to choose $N \subset \rho (A)$. If $\rho  (A) \cap N$ is non-trivial we can quotient out by the connected component of the identity of $\rho (A) \cap N$ and proceed by induction. So assume that $\rho(A) \cap N$ is discrete. The normal connected subgroup $[N,\rho(A)]$ is contained in $\rho (A) \cap N$, so is trivial. So $\rho (A)$ and $N$ commute and, hence, $\rho (A)N$ is an abelian normal subgroup. But $\overline{N\rho (A)}/\rho (A)$ is compact, so $\overline{N\rho (A)}$ is of the form $\rho (A)\times C'$ with $C'$ a compact subgroup centralised by $\rho(L)$. So $N':=(NC') \cap \rho (A)$ is a connected subgroup of $\rho (A)$ such that $N' \subset C'N$ and $N \subset C'N'$. There is therefore a further compact subset $C''$ such that $\rho(\Lambda_{sol}) \subset C''N'$ and $N' \subset C''\rho(\Lambda_{sol})$. Since $\rho_{VA}$ is proper, we find that there is a compact subset $C'''$ such that $\Lambda_{sol} \subset C''' N''$ and $N'' \subset C''' \Lambda_{sol}$ where $N'':=(\rho)^{-1}(N')$.  This is enough to conclude when $G$ is a Lie group.

 Let us return to the general case. By the Gleason--Yamabe theorem \cite[Thm 5']{10.2307/1969792}, there is an open subgroup $U \subset G$ and a compact normal subgroup $K \subset U$ such that $U/K$ is a connected Lie group. Since $U$ is open and $A$ is connected, $A \subset U$. So $KA$ (and, hence,  $\Lambda$) is covered by finitely many translates of $U$. According to Lemma \ref{Lemma: Intersection of commensurable sets}, $\Lambda':=\Lambda^2 \cap U$ is commensurable with $\Lambda$. By the previous paragraph applied to the projection of $\Lambda'$ to $U/K$, we find a vector subspace $V \subset A$ and a compact subset $K' \subset U$ such that $\Lambda' \subset K'V$ and $V \subset K'\Lambda' $.  There is therefore $K'' \subset G$ compact such that  $\Lambda \subset K''V$ and $V \subset K''\Lambda $. Hence, we have proved Lemma \ref{Lemma: Schreiber around subspace}.
\end{proof}
\end{appendix}


\begin{thebibliography}{Mac23b}

\bibitem[BBF16]{BestvinaBrombergFujiwara2016Bounded}
Mladen Bestvina, Ken Bromberg, and Koji Fujiwara.
\newblock Bounded cohomology with coefficients in uniformly convex banach
  spaces.
\newblock {\em Commentarii Mathematici Helvetici}, 91:203--218, 01 2016.

\bibitem[BH18]{bjorklund2016approximate}
Michael Bj\"{o}rklund and Tobias Hartnick.
\newblock {Approximate lattices}.
\newblock {\em Duke Math. J.}, 167(15):2903--2964, 2018.

\bibitem[BHK21]{BjorklundHartnickKarasik}
Michael Björklund, Tobias Hartnick, and Yakov Karasik.
\newblock Intersection spaces and multiple transverse recurrence, 2021.

\bibitem[BHP18]{bjorklund2016aperiodic}
Michael Bj\"{o}rklund, Tobias Hartnick, and Felix Pogorzelski.
\newblock Aperiodic order and spherical diffraction, {I}: auto-correlation of
  regular model sets.
\newblock {\em Proc. Lond. Math. Soc. (3)}, 116(4):957--996, 2018.

\bibitem[BHS19]{bjorklund2019borel}
Michael Bj\"{o}rklund, Tobias Hartnick, and Thierry Stulemeijer.
\newblock Borel density for approximate lattices.
\newblock {\em Forum Math. Sigma}, 7:Paper No. e40, 27, 2019.

\bibitem[Bou81]{MR633754}
Nicolas Bourbaki.
\newblock {\em Espaces vectoriels topologiques. {C}hapitres 1 \`a 5}.
\newblock Masson, Paris, new edition, 1981.
\newblock \'{E}l\'{e}ments de math\'{e}matique. [Elements of mathematics].

\bibitem[CM18]{zbMATH06949681}
Pierre-Emmanuel Caprace and Nicolas Monod.
\newblock Future directions in locally compact groups: a tentative problem
  list.
\newblock In {\em New directions in locally compact groups}, pages 343--356.
  Cambridge: Cambridge University Press, 2018.

\bibitem[CS10]{MR2738997}
Ernie Croot and Olof Sisask.
\newblock A probabilistic technique for finding almost-periods of convolutions.
\newblock {\em Geom. Funct. Anal.}, 20(6):1367--1396, 2010.

\bibitem[Fel62]{MR139135}
J.~M.~G. Fell.
\newblock A {H}ausdorff topology for the closed subsets of a locally compact
  non-{H}ausdorff space.
\newblock {\em Proc. Amer. Math. Soc.}, 13:472--476, 1962.

\bibitem[FM43]{MR7916}
H.~Federer and A.~P. Morse.
\newblock Some properties of measurable functions.
\newblock {\em Bull. Amer. Math. Soc.}, 49:270--277, 1943.

\bibitem[HKP22]{hrushovski2019amenability}
Ehud Hrushovski, Krzysztof Krupi{\'n}ski, and Anand Pillay.
\newblock Amenability, connected components, and definable actions, 2022.
\newblock Id/No 16.

\bibitem[Hoc65]{zbMATH03212078}
G.~Hochschild.
\newblock The structure of {Lie} groups. ({Holden}-{Day} {Series} in
  {Mathematics}).
\newblock San {Francisco}-{London}-{Amsterdam}: {Holde}-{Day}, {Inc}. {IX}, 230
  p. (1965)., 1965.

\bibitem[Hru20]{hrushovski2020beyond}
Ehud Hrushovski.
\newblock Beyond the lascar group.
\newblock {\em arXiv preprint arXiv:2011.12009}, 2020.

\bibitem[Mac22]{machado2019infinite}
Simon Machado.
\newblock Infinite approximate subgroups of soluble {L}ie groups.
\newblock {\em Math. Ann.}, 382(1-2):285--301, 2022.

\bibitem[Mac23a]{machado2020apphigherrank}
Simon Machado.
\newblock Approximate lattices in higher-rank semi-simple groups.
\newblock {\em Geom. Funct. Anal.}, 33(4):1101--1140, 2023.

\bibitem[Mac23b]{machado2019goodmodels}
Simon Machado.
\newblock Closed approximate subgroups: compactness, amenability and
  approximate lattices, 2023.

\bibitem[Mac23c]{mac2023structure}
Simon Machado.
\newblock Structure of approximate lattices in linear groups, 2023.

\bibitem[Mey72]{meyer1972algebraic}
Yves Meyer.
\newblock {\em {Algebraic numbers and harmonic analysis}}, volume~2.
\newblock Elsevier, 1972.

\bibitem[Moo97]{moody1997meyer}
Robert~V Moody.
\newblock {Meyer sets and their duals}.
\newblock {\em NATO ASI Series C Mathematical and Physical Sciences-Advanced
  Study Institute}, 489:403--442, 1997.

\bibitem[MW15]{MR3345797}
Jean-Cyrille Massicot and Frank~O. Wagner.
\newblock Approximate subgroups.
\newblock {\em J. \'{E}c. polytech. Math.}, 2:55--64, 2015.

\bibitem[Rag72]{raghunathan1972discrete}
M.~S. Raghunathan.
\newblock Discrete subgroups of {L}ie groups.
\newblock pages ix+227, 1972.

\bibitem[San12]{MR2911137}
Tom Sanders.
\newblock {Approximate groups and doubling metrics}.
\newblock {\em Math. Proc. Cambridge Philos. Soc.}, 152(3):385--404, 2012.

\bibitem[Sie45]{10.2307/1969027}
Carl~Ludwig Siegel.
\newblock A mean value theorem in geometry of numbers.
\newblock {\em Annals of Mathematics}, 46(2):340--347, 1945.

\bibitem[Yam53]{10.2307/1969792}
Hidehiko Yamabe.
\newblock {A Generalization of A Theorem of Gleason}.
\newblock {\em Annals of Mathematics}, 58(2):351--365, 1953.

\bibitem[Zim13]{zimmer2013ergodic}
Robert~J Zimmer.
\newblock {\em {Ergodic theory and semisimple groups}}, volume~81.
\newblock Springer Science \& Business Media, 2013.

\end{thebibliography}
\end{document}